\documentclass[a4paper,11pt]{article}
\usepackage[utf8]{inputenc}
\usepackage{bm}             

\usepackage{enumerate}
\usepackage{amsmath}
\usepackage{amssymb}
\usepackage{cite}

\newtheorem{Thm}{Theorem}[section]

\newtheorem{Lem}[Thm]{Lemma}
\newtheorem{Cor}[Thm]{Corollary}
\newtheorem{Prop}[Thm]{Proposition}

\newtheorem{Rem}[Thm]{Remark}

\numberwithin{equation}{section}

\newcommand{\Lra}{\Longrightarrow}
\newcommand{\R}{\mathbb{R}}

\newcommand{\cH}{{\mathcal H}}

\newcommand{\cN}{{\mathcal N}}

\newcommand{\weak}{\rightharpoonup}

\newcommand{\eps}{\varepsilon}

\newcommand{\sOmega}{{\mbox{\tiny $\Omega$}}}   
\renewcommand{\phi}{\varphi}

\newenvironment{proof}   
{\noindent  
{\em Proof}.}   
{\nopagebreak\mbox{}\hfill $\Box$\par\addvspace{0.5cm}}   
\newenvironment{altproof}[1]   
{\noindent  
{\em Proof of {#1}}.}   
{\nopagebreak\mbox{}\hfill $\Box$\par\addvspace{0.5cm}}

\DeclareMathOperator{\innt}{int}

\newcommand{\Rb}{{\mathbb{R}}}

\renewcommand{\epsilon}{\varepsilon}

\newcommand{\Lp}{\Lambda_p}
\newcommand{\Lpp}{\Lambda_p^\prime}
\newcommand{\lp}{\lambda_p}

\newcommand{\p}{p}
\newcommand{\q}{q}
\newcommand{\hza}{H_{za}}
\newcommand{\has}{H_{as}}
\renewcommand{\u}{{u}}   
\newcommand{\ux}{u_{x_i}}  
\newcommand{\uxn}{u_{x_N}} 
\newcommand{\up}{{u_p}}
\newcommand{\vp}{v_p}
\newcommand{\wpp}{w_p}
\newcommand{\wt}{\tilde{w}_p}
\newcommand{\uh}{u_{\mbox{\tiny $H$}}} 
\newcommand{\sigmah}{\sigma_{\mbox{\tiny $H$}}}
\newcommand{\sigmaho}{\sigma_{\mbox{\tiny $H_0$}}}
\newcommand{\A}{{\cal A}}
\newcommand{\opm}{(\Omega_{u_p})_{-}}

\newcommand{\upxn}{(u_p)_{x_N}}

\begin{document}
\title{The shape of extremal functions\\ 
for Poincaré-Sobolev-type
inequalities in a ball}
\author{Pedro Girão\\
\small{Mathematics Department, IST, Av. Rovisco Pais, 1049-001 Lisboa 
(Portugal)\thanks{Partially supported by FCT/POCTI/FEDER.}
}
\and 
Tobias Weth\\
\small{Math. Institut, Universit\"at Giessen, Arndtstr. 2, 
35392 Giessen (Germany)}
}

\date{}
\maketitle

\begin{abstract}
\noindent We study extremal functions for a family of 
Poincaré-Sobolev-type inequalities. These functions minimize, for subcritical or critical $p\geq 2$, the quotient ${\|\nabla u\|_2}/{\|u\|_p}$ among all $u \in H^1(B)\setminus\{0\}$ with $\int_{B}u=0$. Here $B$ is the unit ball in $\R^N$. We show that the minimizers are axially symmetric with respect to 
a line passing through the origin. We also show that they are strictly monotone in the direction of this line. In particular, they take their maximum and minimum precisely at two antipodal points on the boundary of $B$. We 
also prove that, {\em for $p$ close to $2$}, minimizers are 
antisymmetric with respect to the hyperplane through the origin perpendicular to 
the symmetry axis, and that, once the symmetry axis is fixed,
they are unique (up to multiplication by a constant).
In space dimension two, we prove that minimizers are not antisymmetric for large $p$.
\end{abstract}

\section{Introduction and main results}
\label{sec:introduction}
Let $\Omega \subset \R^N$ be bounded domain with smooth 
boundary. Moreover, let $q \ge 1$; let $1 \le \p \le \frac{qN}{N-\q}$ if 
$N> \q$, $1 \le \p<\infty$ if $N=q$, $1 \le \p\leq\infty$ if $N<q$. We consider the family of 
Poincar\'e-Sobolev-type inequalities 
\begin{equation}
\label{eq:11}
\biggl(\int_\Omega |u-u_\sOmega|^{\p}\biggr)^{\frac{\q}{\p}} \le 
C(\p,\q,\Omega) \int_\Omega |\nabla u|^\q \qquad \forall u \in 
W^{1,q}(\Omega),
\end{equation}
where $u_\sOmega = \frac{1}{|\Omega|}\int_\Omega u$ is the average of $u$ on $\Omega$. This family of inequalities can be derived by combining Poincar\'e's inequalities with Sobolev embeddings, see e.g. \cite[Section 3.6]{giusti-new}. But this derivation neither yields optimal constants $C(\p,\q,\Omega)$, nor it answers the question whether equality
can be achieved and, if so, how 
extremal functions look like for 
particular domains $\Omega$. These questions, which are of interest both from an analytical and a geometrical point of view, have been addressed in a number of papers, but answers have only been obtained  
in special cases so far. In the `linear' case 
$p=q=2$, the best constant $C(2,2,\Omega)$ is just the inverse 
of the second 
eigenvalue 
$\lambda_2(\Omega)$ of the Neumann Laplacian on the domain 
$\Omega$, and for $u \not=0$ equality holds in (\ref{eq:11}) if and only if $u$ is a corresponding eigenfunction. For some domains, $\lambda_2(\Omega)$ and its eigenspace can be computed in terms of special functions. A general upper estimate for 
$\lambda_2(\Omega)$ is given by an isoperimetric inequality due to 
Szeg\"o \cite{szegoe} for $N=2$ and Weinberger \cite{Weinberger} for $N \ge 3$. This inequality states that, among all domains of fixed volume, $\lambda_2(\Omega)$ is 
maximal for the ball. 
For convex domains, a lower estimate for $\lambda_2(\Omega)$ is given in \cite{payne.weinberger} in terms of the diameter of $\Omega$, and in the 
two-dimensional case the location of the nodal line is studied in \cite{jerison}. The case $q=1$ also received much attention. In this case, the best constant in 
(\ref{eq:11}) is attained in the space of functions of 
bounded variation,  
and the extremal functions directly reflect 
geometric properties of the domain $\Omega$, see \cite{maz'ya,Z3}.
 
The present paper is motivated by the rather complete 
description obtained recently for the {\em one-dimensional case}, i.e., for $\Omega= (-1,1) \subset \R$. In this case, building upon previous work of Dacorogna-Gangbo-Subía \cite{DGS}, Egorov \cite{egorov-alt},
Buslaev-Kontratiev-Nazarov~\cite{bus-kon-naz},
Belloni-Kawohl \cite{belloni-kawohl} and 
Kawohl \cite{kawohl1},  Nazarov \cite{nazarov-new} completed the proof of the following result.

\begin{Thm} {\rm (see \cite{nazarov-new})}
\label{sec:introduction-4}
Let $\Omega = (-1,1)$, and let $p, q \in (1,\infty)$. Then the best constant $C(\p,\q,\Omega)$ in {\rm (\ref{eq:11})} is attained, and the corresponding extremal 
functions are either strictly increasing or strictly decreasing on $(-1,1)$. Moreover, for $\p \le 3\q$, 
the best constant is attained by an odd function $u_{\p,\q}$, and every other extremal function with 
$\int_{-1}^1 u=0$ is a scalar multiple of $u_{\p,\q}$. For $\p>3\q$, the extremal functions are not odd.   
\end{Thm}

The proof of this theorem is based on ordinary differential equation techniques. A crucial fact which is used is the 
existence of a first integral for the corresponding Euler equation.  
In this paper we study the case of multidimensional domains 
$\Omega \subset \R^N,\: N \ge 2$, which requires a 
new approach. We focus on the case 
$q=2$ and $2\le p \le 2^*$ for $N \ge 3$, $2 \le p < \infty$ for $N=1,2$, where $2^*=2N/(N-2)$ is the critical Sobolev exponent. For $p>2$, not much seems to be known about extremal functions even on simple domains. The only result we are 
aware of is concerned with a rectangle in $\R^2$, see \cite{nazarov-2d}. Let $\|u\|_p$ denote the usual $L^p$-norm of a function $u \in L^p(\Omega)$, and let $\hza(\Omega)$ denote the space of all functions $u \in H^1(\Omega)$ with $\int_\Omega u = 0$, endowed with the norm 
$\|\nabla u\|_2$. Then the best 
constant $C(p,2,\Omega)$ in (\ref{eq:11}) is just 
the square of the norm of the embedding $\hza(\Omega) \hookrightarrow L^p(\Omega)$, and it is the inverse of the number 
\begin{equation}
\label{eq:12}
\Lp(\Omega)=\inf_{u\in \hza(\Omega) \setminus \{0\}}\frac{\|\nabla u\|^2_2}{\|u\|_p^2}.
\end{equation}
For subcritical $p<2^*$, one may use the compactness of the embedding $H^1(\Omega) \hookrightarrow L^p(\Omega)$ to show that the minimum in (\ref{eq:12}) is attained. We first extend this statement to the critical case $N \ge 3$, $p= 2^*$ where compactness fails for the embedding $H^1(\Omega) \hookrightarrow L^p(\Omega)$. We let, as usual, $S$ stand for the best Sobolev constant, i.e.,
\begin{equation}
\label{eq:9}
S=\inf_{u \in C_0^\infty(\R^N)\setminus\{0\}}\frac{\|\nabla u\|_2^2}{\|u\|_{2^*}^2}.
\end{equation}

\begin{Prop}
\label{sec:exist-minimz-crit-1}
Let $N \ge 3$, $p=2^*= \frac{2N}{N-2}$, and let $\Omega \subset \R^N$ be a smooth bounded domain. Then $\Lp(\Omega) < 
\frac{S}{2^{2/N}}$, and the minimum in {\rm (\ref{eq:12})}
is achieved.   
\end{Prop}
In the proof of this observation, estimates for critical exponent Neumann problems due to Adimurthi-Mancini \cite{adimurthi.mancini} and Wang \cite{wang} play a crucial role. The main goal of this paper is to analyze the shape of minimizing functions, aiming for similar results as obtained in Theorem~\ref{sec:introduction-4} for the case of an interval. We note that every 
normalized minimizer $\u \in \hza(\Omega) \setminus \{0\}$, $\|\nabla u\|_2=1$, of (\ref{eq:12}) is a sign changing weak solution of 
the problem
\begin{equation}
\label{eq:37}
-\Delta\u=\lp\,|\u|^{p-2}\u+\mu_p\quad {\rm in}\ \Omega,\qquad
\textstyle\frac{\partial\u}
{\partial\nu}=0\quad {\rm on}\ \partial \Omega.
\end{equation}
Here $\lp= [\Lp(\Omega)]^{p/2}$, and $\mu_p$ is given by
\begin{equation}
\label{eq:8}
\mu_p=\mu_p(u)=-\frac{\lp}{|\Omega|}\int_\Omega|\u|^{p-2}\u.
\end{equation}
By elliptic regularity theory, $\u\in C^{3,\alpha}(\overline{\Omega})$ for some $0<\alpha<1$.
We focus on the case where the domain is the open unit ball $B \subset \Rb^N$, but we also discuss the case of an annulus and some extensions to nonradial domains, see Theorems \ref{sec:antisymmetry-p-close-2}, \ref{sec:antisymm-break-large-4} and Section~\ref{sec:annulus}.
Our main results are collected in the following theorem.

\begin{Thm}
\label{sec:introduction-1}
Let $2\leq p \le 2^*$ for $N \ge 3$, $2\leq p< \infty$ for $N=2$, and let $\u$ be a minimizer for {\rm (\ref{eq:12})} on the unit ball $B \subset \R^N$. Then there exists a unit vector $e \in \R^N$ such that
\begin{itemize}
\item[{\rm (a)}] $\u(x)$ only depends on $r=|x|$ and $\theta:= 
\arccos\bigl(\frac{x}{|x|}\cdot e\bigr)$. Hence $\u$ is axially symmetric with respect to the axis passing through $0$ and $e$.
\item[{\rm (b)}] $\frac{\partial \u}{\partial \theta}(r,\theta)<0$ 
for $0<r \le 1,\:0<\theta <\pi$. 
\item[{\rm (c)}] $\partial_e \u > 0$ on $\overline{B}\setminus \{\pm e\}$.
If $\tau$ is another unit vector in $\Rb^N$ orthogonal to $e$, then $\partial_\tau u$ has precisely four nodal domains.
Here $\partial_e$ and $\partial_\tau$ denote the directional derivatives in the direction of $e$ and $\tau$, respectively.
\item[{\rm (d)}] If $p$ is close to $2$, then $\u$ is antisymmetric with respect to the reflection $x \mapsto x- 2 (x \cdot e)e$ 
at the hyperplane $H_e:=\{x \in \R^N:\: x \cdot e = 0\}$.
Furthermore, if $p$ is close to $2$, then every other minimizer of {\rm (\ref{eq:12})} whose axis of symmetry has direction
$e$ is a scalar multiple of $u$.
\item[{\rm (e)}] In the two dimensional case $N=2$, the function $\u$ is not antisymmetric when $p$ is sufficiently large.   
\end{itemize}
\end{Thm}
When $u: B \to \R$ is a function defined on the unit ball $B$ which only depends on $r=|x|$ and $\theta:= 
\arccos\bigl(\frac{x}{|x|}\cdot e\bigr)$ for some fixed 
$e \in \partial B$, then we freely vary between the notations 
$u(x),\: x \in B$ and $u(r,\theta),\: 0 \le r \le 1,\: \theta \in [0,\pi]$.

\begin{Rem}
\label{sec:introduction-2}
{\rm 
(i) In the case $p=2$, minimizers of (\ref{eq:12}) are precisely the eigenfunctions of the Neumann 
Laplacian on $B$ corresponding to the first nontrivial eigenvalue $\Lambda_2$. For these eigenfunctions, properties (a)-(d) can be verified easily, see Section~\ref{sec:case-p=2}.
  
(ii) Properties 
(a) and (b) imply that $u$ is {\em foliated Schwarz symmetric} in the sense of \cite{SW,BWW}. In \cite{kawohl} this symmetry is called {\em spherical symmetry}, whereas in \cite{brock} it is called {\em codimension-one symmetry}. 
    
(iii) By properties (b) and (c), $\u$ takes its maximum and minimum precisely at the two antipodal points $\{\pm e\}$ on the boundary of $B$
and has precisely two nodal domains. In particular, $u$ is a nonradial function. At first glance, one might guess that (c) follows 
from a monotone rearrangement along straight lines. However, it is unclear whether the 
Dirichlet integral in the numerator of (\ref{eq:12}) decreases under this 
rearrangement, see \cite[Remark~2.36]{kawohl}. Our proof follows a different approach described below.

(iv) In the case that $u$ is antisymmetric with respect to the reflection  
at the hyperplane $H_e$, the four nodal domains of $\partial_\tau u$ are the four quadrants in $B$
cut off by the hyperplanes
$H_e$ and $H_\tau:=\{x \in \R^N:\: x \cdot\tau = 0\}$.

(v) Part (d) and (e) show that the `antisymmetry breaking' observed in dimension one (see Theorem~\ref{sec:introduction-4} above) also occurs in the two-dimensional case for $p$ somewhere strictly between $2$ and $\infty$. It would be interesting to have more information about the precise value where the symmetry breaking occurs. In dimensions $N \ge 3$, we 
do not know whether for any $p$ there exist minimizers which are not antisymmetric. 
 
(vi) Part (d) and (\ref{eq:8}) yield 
$\mu_p=\mu_p(u)=0$ for $p$ close to $2$ and any minimizer $\u$ in (\ref{eq:12}), hence $\u$ solves an equation 
with a homogeneous right hand side. 

(vii) In the case that $p=2$, 
$u(r,\theta)=g(r)\cos\theta$ for some function $g:[0,1]\to\R$. So,
it seems natural to ask if
there exist functions $R:[0,1]\to\R$
and $\Theta:[0,\pi]\to\R$ such that $u(r,\theta)=R(r)\Theta(\theta)$, for $p>2$. 
We will show that this is not the case for $u$ antisymmetric (see Remark~\ref{sec:strict-monot-axial-5}). 

(viii) For an annulus $\A=\{x\in\R^N:\:\rho<|x|<1\}$, $0<\rho<1$, analogues of (a), (b), (d) and (e) hold.
With regard to (c) we only have a partial result, see Section~\ref{sec:annulus} below.

(ix) Part (a) of Theorem~\ref{sec:introduction-1} is also true for $1<p\le 2$, see Section~\ref{sec:axial-symm-minim} below. It would be interesting to know whether parts (a)-(c) also hold for the general quasilinear case $q \not= 2$ and all $1<p < \frac{qN}{N-q}$.  
Most of the arguments in the present paper 
use the fact that minimizers solve a semilinear 
elliptic equation with an increasing 
$C^1$-nonlinearity, so they require $\q=2$, $\p\ge 2$.}  
\end{Rem}

We mention further work related to our results. 
On Riemannian manifolds, the Poincaré-Sobolev inequality has been studied by Zhu \cite{Z1,Z2,Z3}. In 
\cite{Z1} he proves the existence of extremal functions for (\ref{eq:11}) on the standard 
$N$-dimensional
sphere for $\q \in (1,(1+\sqrt{1+8N})/4)$ and critical $\p= \frac{qN}{N-\q}$. In \cite{Z3} he proves  interesting geometric results for the case $\q=1,\p=2$ on a 
two dimensional Riemannian manifold. Recently, Bartsch, Willem and one of the authors proved in \cite{BWW} that least energy sign changing solutions of a superlinear problem similar to (\ref{eq:37}) on a ball or an annulus are axially symmetric. 
Using this information and a result from \cite{bartsch.weth1}, Aftalion and Pacella \cite{AP3} then deduced further 
properties of these solutions, showing in particular that least energy solutions are not radially symmetric. 
However, in contrast to the present paper, only Dirichlet boundary conditions were considered in \cite{BWW,AP3}, and in \cite{AP3} this seems to enter crucially in the proofs. Extremal functions for the {\em trace Sobolev inequality} in a ball have been
determined by Carlen and Loss in \cite{CL} exploiting conformal
invariance and rearrangements of functions; and they were also obtained
later by Maggi and Villani in~\cite{MV} using mass transportation
methods, in the spirit of Cordero-Erausquin, Nazaret and 
Villani~\cite{CNV}.

Next we briefly describe the techniques we use to prove Theorem~\ref{sec:introduction-1}. 
The proof of~(a) follows the ideas in \cite{BWW}. It uses a different characterization of the desired symmetry property 
by simple two point rearrangement inequalities corresponding to the family of hyperplanes which contain the origin. The corresponding two point rearrangement is called polarization, and it is 
also used for instance in \cite{brock,brock.solynin:2000,baernstein.taylor}.    

For the proof of (b) and (c), in the case that $p>2$, we first reformulate the minimization problem (\ref{eq:12}) in terms of a 
non-homogeneous functional $G: \hza(B) \to \R$ whose second derivative is easier to 
study. The minimizers of (\ref{eq:12}) then correspond to minimizers of the restriction of  
$G$ to the associated Nehari manifold. We then
investigate properties of the directional derivatives $\frac{\partial u}{\partial x_i}$ of $u$. These functions are easily seen to be pointwise solutions of the linearized problem, but they do not satisfy homogeneous Neumann boundary conditions. Nevertheless we can use these functions to show the asserted monotonicity properties at least in 
certain subregions of the unit ball $B$. 
The proof is then completed by a moving plane argument.
We feel that this combined approach 
has further applications for problems with Neumann or mixed boundary conditions.
   
Part (d) is proved by a perturbation argument based on the fact that, as $p \to 2$, the minimizers of (\ref{eq:12}) approach eigenfunctions of the Neumann 
Laplacian on $B$ corresponding to the second (hence the first nontrivial) eigenvalue. This eigenvalue is degenerate, 
but we can remove this degeneracy by fixing 
some axis of symmetry. 
 In this fixed space of axially symmetric functions, there is 
only a one-dimensional subspace of corresponding 
eigenfunctions, and these eigenfunctions are antisymmetric. 
Somewhat similar perturbation arguments have been used 
by Dancer~\cite{dancer} and Lin~\cite{lin} to prove uniqueness of {\em positive} solutions for some slightly superlinear 
Dirichlet problems. 

The proof of (e) relies on the facts that, in dimension two,
$\Lambda_p$ converges to $0$ as $p \to \infty$, and the same is true
when the infimum in (\ref{eq:12}) is taken in the class of antisymmetric
functions.

The paper is organized as follows. In Section~\ref{sec:exist-minimz-crit} we prove Proposition~\ref{sec:exist-minimz-crit-1}. 
In Section \ref{sec:case-p=2} we briefly recall how Theorem~\ref{sec:introduction-1}(a)-(d) can be derived in the linear case
$p=2$. Section~\ref{sec:axial-symm-minim} is devoted to the proof of Theorem~\ref{sec:introduction-1}(a) and a weak form of the 
monotonicity property claimed in part (b). The proof of parts (b) 
and (c) are completed in Section~\ref{sec:strict-monot-axial}. In Section~\ref{sec:antisymmetry-p-close} we consider $p$ close to $2$ and prove Theorem~\ref{sec:introduction-1}(d). The proof of part (e) is contained in Section~\ref{sec:antisymm-break-large}. 
The case of an annulus is discussed in Section~\ref{sec:annulus}.
Finally, the appendix is devoted to the Hopf boundary lemma which plays a crucial role in our arguments. Here we prove a version for half-balls with a slightly stronger conclusion as usually stated in the literature.    

Throughout the paper, whenever the underlying domain is the unit ball $B$, we will just write $\hza$ instead of $\hza(B)$, $\Lp$ instead of $\Lp(B)$, etc..
Finally, if $A$ is a subset of $\R^N$, we denote by $\innt(A)$, $\overline A$, and $\partial A$ the interior, 
closure, and boundary of $A$, respectively.

\textbf{Acknowledgement:} The second author would like to thank Norman Dancer, Massimo Grossi and Filomena Pacella for helpful discussions.

\section{Existence of minimizers in the critical case}
\label{sec:exist-minimz-crit}
In this section we prove Proposition~\ref{sec:exist-minimz-crit-1}.
The proof relies on the following estimate. 

\begin{Prop}
\label{sec:exist-minim-crit}
Let $N \ge 3$, $p=2^*= \frac{2N}{N-2}$, and let $\Omega \subset \R^N$ be a smooth bounded domain. Then $\Lp(\Omega)<
\frac{S}{2^{2/N}}$, where $S$ 
is defined in {\rm (\ref{eq:9})}. 
\end{Prop}

\begin{proof}
Without loss of generality, we may assume that $0 \in \partial \Omega$, and 
that the mean curvature of $\partial \Omega$ at $0$ is strictly positive. We consider the Aubin-Talenti instantons $u_\eps \in 
H^{1}(\Omega),\:\eps>0$, restricted to the domain $\Omega$. These are defined by
$$
u_\eps(x)= \frac{ [N(N-2)\eps^2]^{\frac{N-2}{4}} }
{[\eps^2+|x|^2]^{\frac{N-2}{2}} }= \eps^{\frac{2-N}{2}}U\left(\frac{x}{\eps}\right),\qquad x \in \Omega,
$$  
where $U(x):= \frac{[N(N-2)]^{\frac{N-2}{4}}}{[1+|x|^2]^{\frac{N-2}{2}}}$. Then, as $\eps \to 0$, we have the following estimates due to Adimurthi-Mancini (see \cite[Proof of Lemma 2.2]{adimurthi.mancini}):
\begin{align}
\frac{\|\nabla u_\eps\|_2^2}{\|u_\eps\|_{p}^2} & \le \frac{S}{2^{2/N}} 
\Bigl(1-c_0\eps |\log \eps| + O(\eps)\Bigr) \qquad \text{if }N= 3,\label{eq:26}\\  
\frac{\|\nabla u_\eps\|_2^2}{\|u_\eps\|_{p}^2}& \le \frac{S}{2^{2/N}} 
\Bigl(1-c_1 \eps + O(\eps^2|\log \eps|)\Bigr) \qquad \text{if }N= 4,\label{eq:27}\\     
\frac{\|\nabla u_\eps\|_2^2}{\|u_\eps\|_{p}^2}& \le \frac{S}{2^{2/N}} 
\Bigl(1-c_2\eps + O(\eps^2)\Bigr) \qquad \text{if }N \ge 5\label{eq:28}.     
\end{align}
Above and in the following, $c_0,c_1,c_2,\dots$ are positive constants which may depend on the dimension $N$. Moreover, for $s \ge 1$ we have
\begin{align*}
\|u_\eps\|_s^s &= \int_\Omega u_\eps^s \:dx \le \int_{B_R(0)}u_\eps^s \:dx 
= \eps^{\frac{s(2-N)}{2}} 
\int_{B_R(0)} U^s\left(\frac{x}{\eps}\right)\:dx \\
&= \eps^{\frac{s(2-N)}{2}+N} 
\int_{B_{\frac{R}{\eps}}(0)}U^s(z)\:dz = c_3 \eps^{\frac{s(2-N)}{2}+N} 
\int_{0}^{\frac{R}{\eps}}\frac{r^{N-1}}{(1+r^2)^{\frac{s(N-2)}{2}}}\:dr \\
&\le \eps^{\frac{s(2-N)}{2}+N}\left( c_4 +c_5 
\int_{1}^{\frac{R}{\eps}}r^{(N-1)-s(N-2)}\:dr \right)\\
&= \eps^{\frac{s(2-N)}{2}+N}\Bigl ( c_4 +c_6\eps^{s(N-2)-N} 
\Bigr)\\
&= O\left(\eps^{\min\left\{\frac{s(2-N)}{2}+N,\frac{s(N-2)}{2}\right\}}\right).   
\end{align*}
Here $R>0$ is chosen so large such that 
$\Omega \subset B_R(0)$. In particular,
$$
\|u_\eps\|_1 = O\bigl(\eps^{\frac{N-2}{2}}\bigr)\qquad \text{and}\qquad  
\|u_\eps\|_{p-1}^{p-1}=O\bigl(\eps^{\frac{N-2}{2}}\bigr).
$$ 
We set $v_\eps = u_\eps -a_\eps$, where $a_\eps= (1/|\Omega|)\int_\Omega u_\eps$. Then $v_\eps \in \hza(\Omega)$ and
$a_\eps= O\bigl(\eps^{\frac{N-2}{2}}\bigr)$. We recall that there is $C=C(p)>0$ such that
$$
\bigl ||a+b|^p -|a|^p-|b|^p \bigr| \le C 
\bigl(|a|^{p-1}|b|+|a| |b|^{p-1} \bigr) \qquad \text{for all}\quad a,b \in \R.
$$ 
So, we estimate 
\begin{align*}
\int_\Omega |v_\eps|^p\:dx &= 
\int_\Omega |u_\eps-a_\eps|^p\:dx \\
&\ge \int_\Omega\Bigl(|u_\eps|^p + a_\eps^p 
- C [|u_\eps|^{p-1}a_\eps + |u_\eps|a_\eps^{p-1}]\Bigr)\:dx \\
&= \int_\Omega |u_\eps|^p\:dx + |\Omega|a_\eps^p 
- C \left(a_\eps \int_\Omega |u_\eps|^{p-1}\:dx 
+ a_\eps^{p-1}\int_\Omega |u_\eps|\:dx \right) \\
&\ge \int_\Omega |u_\eps|^p\:dx -O\bigl(\eps^{N-2}\bigr).
\end{align*}
Consequently,
$$
\|v_\eps\|_p^2 \ge \|u_\eps\|_p^2-O\bigl(\eps^{N-2}\bigr),
$$
and therefore
$$
\frac{\|\nabla v_\eps \|_2^2}{\|v_\eps \|_p^2}\le 
\frac{\|\nabla u_\eps \|_2^2}{\|u_\eps\|_p^2- 
O(\eps^{N-2})}=\frac{\|\nabla u_\eps \|_2^2}{\|u_\eps\|_p^2}+ 
O\bigl(\eps^{N-2}\bigr).  
$$
Combining this with (\ref{eq:26})--(\ref{eq:28}), we obtain
$$
\frac{\|\nabla v_\eps \|_2^2}{\|v_\eps \|_p^2} <
\frac{S}{2^{2/N}}\qquad \text{for $\eps$ small enough,}
$$
and hence $\Lp(\Omega) < \frac{S}{2^{2/N}}$.
\end{proof}

\begin{altproof}{{\rm \ Proposition~\ref{sec:exist-minimz-crit-1}} (completed)} We consider a minimizing sequence 
$(u_n)\in \hza(\Omega)$ for (\ref{eq:12}), which we can normalize such that  
$\|\nabla u_n\|_2^2=\Lambda_p$ for all $n$. Hence
$$
{\|u_n\|_p^{2}} \to 1\qquad \text{as $n \to \infty$}.
$$
We may pass to a subsequence such that
\begin{align*}
u_n &\weak u &&\qquad \text{weakly in }\hza(\Omega)\\
u_n &\to u &&\qquad \text{strongly in $L^s(\Omega)$ for 
$s<p=2^*$}\\
u_n(x) &\to u(x) &&\qquad \text{for a.e. $x \in \Omega$}.
\end{align*}
By the Brezis-Lieb Lemma \cite{brezis.lieb},
$$
\|u_n\|_p^p= \|u_n-u\|_p^p + \|u\|_p^p +o(1),
$$
so that
$$
\limsup_{n \to \infty}\Bigl ( \|u_n-u\|_p^2 + \|u\|_p^2 \Bigr)\ge 1,
$$
where equality holds if and only if $u=0$ or $u_n \to u$ strongly in 
$L^p(\Omega)$. But, by definition of $\Lambda_p$, 
$$
\|u_n-u\|_p^2+\|u\|_p^2 \le \frac{ 
\|\nabla (u_n-u)\|_2^2 + \|\nabla u\|_2^2}{\Lambda_p}= 
\frac{\|\nabla u_n\|_2^2+o(1)}{\Lambda_p} =1+o(1). 
$$
Hence we conclude that either $u=0$ or $u \not= 0$ and 
$u_n \to u$ strongly in $L^p(\Omega)$. The first case
can be excluded with the help of Proposition~\ref{sec:exist-minim-crit} and Cherrier's inequality \cite{cherrier}. Indeed, 
Cherrier's inequality states that, for every $\eps>0$, there is a constant $M_\eps$ such that 
$$
\Big(\frac{S}{2^{2/N}}-\eps \Bigr)\|u\|_p^2 \le  
\|\nabla u\|_2^2 + M_\eps \|u\|_2^2.
$$
We choose $\eps= \frac{1}{2}(\frac{S}{2^{2/N}}-\Lp(\Omega))$, which is positive by Proposition~\ref{sec:exist-minim-crit}. Then we get
$$
\Big(\frac{S}{2^{2/N}}-\eps \Bigr)\|u_n\|_p^2 \le  
\|\nabla u_n\|_2^2 + M_\eps \|u_n\|_2^2 \qquad \text{for all $n$}. 
$$
If $u=0$, then $u_n \to 0$ in $L^2(\Omega)$, and thus
$$
\Big(\frac{S}{2^{{2}/{N}}}-\eps \Bigr) 
\le \lim_{n \to \infty} \frac{\|\nabla u_n\|_2^2}{\|u_n\|_p^2} 
= \Lambda_p,
$$
contrary to the choice of $\eps$. We conclude that $u \not= 0$ and $u_n \to u$ in $L^p(\Omega)$, so that
$$
\frac{\|\nabla u\|_2^2}{\|u\|_p^2} \le \lim_{n \to \infty} 
\frac{\|\nabla u_n\|_2^2}{\|u_n\|_p^2} 
= \Lambda_p.
$$ 
Hence $u$ is a minimizer of (\ref{eq:12}). The proof is finished.
\end{altproof}

\section{The case $p=2$}
\label{sec:case-p=2}
Henceforth (except for Theorems \ref{sec:antisymmetry-p-close-2} and \ref{sec:antisymm-break-large-4}, and Section~\ref{sec:annulus}) we focus on the case where the underlying domain is the open unit ball $B \subset \R^N$ centered at zero. In this section we briefly recall some known facts about minimizers of {\rm (\ref{eq:12})} in the `linear' case $p=2$, thus verifying Theorem~\ref{sec:introduction-1}(a)-(d) in this special case. The minimizers are eigenfunctions of the Laplacian with
homogeneous Neumann boundary conditions corresponding to
the first nonzero eigenvalue $\Lambda_2$.  
It is well known that the eigenspace corresponding to 
$\Lambda_2$ is $N$-dimensional, 
and that every eigenfunction can be written as 
\begin{equation}
\label{eq:105}
u(x)=g(r)\Bigl( \frac{x}{r}\cdot e \Bigr) = g(r) \cos \theta  
\end{equation}
for some unit vector $e \in \R^N$, see e.g. \cite{Weinberger}. 
Here $r=|x|$, $\theta= \arccos (\frac{x}{r} \cdot e)$, and $g:[0,1]\to\Rb$ is the
(up to a positive constant) unique solution of the problem
\begin{equation}
\label{eq:103}
g''+\frac{N-1}{r}g'+\Bigl(\Lambda_2-\frac{N-1}{r^2}\Bigr)g=0,\quad g>0 \quad\text{in}\ (0,1],\qquad g(0)=0,\quad g'(1)=0.
\end{equation}
Hence $u$ is axially symmetric with respect to the axis passing through $0$ and $e$. 
Moreover,
$$
\frac{\partial u}{\partial\theta}(r,\theta)=-g(r)\sin\theta<0\qquad\text{for}\quad 0<r\leq 1,\ 0<\theta<\pi,
$$
so that assertions (a), (b) and (d) of Theorem~\ref{sec:introduction-1} hold for $p=2$. To verify (c), we note that $g$ is strictly 
increasing, since $g'$ only vanishes at the 
point $1$. Without loss of generality, we assume $e=e_N=(0,\dots,0,1)$. Then
$$
\partial_e u = \frac{\partial u}{\partial x_N}=
\frac{\partial}{\partial x_N}\left[g(r)\frac{x_N}{r}\right]=
\frac{g'(r)x_N^2}{r^2}-\frac{g(r)x_N^2}{r^3}+\frac{g(r)}{r}=\frac{g'(r)x_N^2}{r^2}+\frac{g(r)(r^2-x_N^2)}{r^3}.
$$
Since both $g$ and $g'$ are positive in $(0,1)$,
$\partial_e u>0$ in $\overline{B}\setminus\{\pm e_N,0\}$.  Also, $\partial_e u(0)=g'(0)>0$.
Now let $\tau$ be another unit vector in $\R^N$ orthogonal to $e$. We may suppose, without loss of generality, that $\tau=e_1$. We compute
\begin{equation*}
\partial_\tau u=
\frac{\partial u}{\partial x_1}=\frac{\partial}{\partial x_1}\Bigl[g(r)\frac{x_N}{r}\Bigr]=\frac{x_1x_N}{r}\frac{d}{dr}
\Bigl[\frac{g(r)}{r}\Bigr]
=\frac{x_1x_N}{r^{N+2}}\Bigl[r^Ng'(r)-r^{N-1}g(r)\Bigr]=
\frac{x_1x_N}{r^{N+2}}f(r),
\end{equation*}
where $f:\:]0,1]\to\R$ is defined by
$f(r):=r^Ng'(r)-r^{N-1}g(r)$.  Using (\ref{eq:103}),
$$f'(r)=r^N\Bigl(g''(r)+\frac{N-1}{r}g'(r)-\frac{N-1}{r^2}g(r)\Bigr)=-\Lambda_2r^Ng(r)<0.$$
From  (\ref{eq:105}) with $\theta=0$, $g'$ is bounded, 
so $\lim_{r\to 0}f(r)=0$.
We deduce that $f$ is negative, so that the nodal domains of $\frac{\partial u}{\partial x_1}$ are precisely the four quadrants in $B$ cut off by the hyperplanes $\{x\in\R^N:\:x_1=0\}$ and $\{x\in\R^N:\:x_N= 0\}$.
This finishes the proof of Theorem~\ref{sec:introduction-1}(c) for $p=2$. 

\section{Axial symmetry of minimizers}
\label{sec:axial-symm-minim}
Let $1< p \le 2^*$ for $N \ge 3$, $1< p< \infty$ for 
$N=2$. Solely in this section we allow values $1<p<2$; for these 
values of $p$ we only have the guarantee that the solutions of (\ref{eq:37}) belong to $C^{2,\alpha}$,
as opposed to belonging to
$C^{3,\alpha}$ for $2\leq p\leq 2^*$. We have the following symmetry result.

\begin{Prop}
\label{sec:axial-symm-minim-1}
Let $\u$ be a minimizer for {\rm (\ref{eq:12})} on $B$. Then $\u$ is foliated Schwarz symmetric, 
i.e.\ there exists a unit vector $e \in \R^N,\:|e|=1$ such that $\u(x)$ only 
depends on $r=|x|$ and $\theta:= \arccos\bigl(\frac{x}{|x|}\cdot e\bigr)$, and $u$ is nonincreasing in $\theta$. Moreover, either
$u$ does not depend on $\theta$ (hence it is a radial function), or $\frac{\partial \u}{\partial \theta}(r,\theta)<0$ for 
$0<r \le 1,\:0<\theta<\pi$. 
\end{Prop}

Let ${\cal H}$ be the family of closed
half-spaces $H$ in $\Rb^N$ such that $0$ lies in the hyperplane $\partial H$.
For $H\in{\cal H}$, we denote by $\sigmah:\Rb^N\to\Rb^N$ the reflection with respect to $\partial H$. We start with the following lemma.

\begin{Lem}
\label{sec:axial-symm-minim-2}
Let $\u$ be a minimizer for {\rm (\ref{eq:12})} on $B$. Let $H\in{\cal H}$, and let $h= h(x)$ denote the outward normal for $x \in \partial H$. Then one of the following holds.
\begin{itemize}
\item[{\rm (i)}] $u(x)>u(\sigmah(x))$ for all $x \in \overline{B} \cap \innt(H)$, and $\frac{\partial u}{\partial h}<0$ on $ \partial H \cap \overline{B}$,
\item[{\rm (ii)}] $u(x)<u(\sigmah(x))$ for all $x \in \overline{B} \cap \innt(H)$, and $\frac{\partial u}{\partial h}>0$ on $\partial H \cap \overline{B}$,
\item[{\rm (iii)}] $u(x)= u(\sigmah(x))$ for all $x \in \overline{B}$.
\end{itemize}
\end{Lem}

\begin{proof}
Without loss of generality, we may assume that $\|\nabla u\|_2=1$.
As in \cite{BWW} we denote by $\uh:\overline{B}\to\Rb$,
$$\uh(x)=\left\{
\begin{array}{ll}
\max\{\u(x),\u(\sigmah(x))\},&x\in H\cap\overline{B},\\
\min\{\u(x),\u(\sigmah(x))\},&x\in\overline{B}\setminus H,
\end{array}\right.
$$
the polarization of $\u$ with respect to $H$.
By Lemma 2.2 of \cite{BWW},
\begin{equation}
\label{eq:29}
\int_B\uh=0,\qquad
\int_B|\uh|^{p-2}\uh=\int_B|\u|^{p-2}\u,\qquad
\int_B|\uh|^{p}=\int_B|\u|^{p},
\end{equation}
while, by Proposition 2.3 of \cite{SW},
$$
\int_B|\nabla\uh|^2=\int_B|\nabla\u|^2.
$$
Hence $\uh$ is also a minimizer of (\ref{eq:12}) on $B$,
and thus it is a weak (and therefore $C^{2,\alpha}$) solution of
$$-\Delta\uh=\lp |\uh|^{p-2}\uh+ \mu_p\ {\rm in}\ B,\quad
\textstyle\frac{\partial\uh}
{\partial\nu}=0\ {\rm on}\ \partial B.$$
By (\ref{eq:8}) and (\ref{eq:29}), we have $\mu_p=\mu_p(\uh)= \mu_p(u)$. Following \cite{BWW}, we consider 
$$
w:\overline{B \cap H} \to \Rb,\qquad w:=|\u-\u\circ\sigmah|=2\uh-(\u+\u\circ\sigmah).
$$
Then $w \in C^{2,\alpha}(\overline{B \cap H})$ (since $\u,\uh, 
u \circ \sigmah \in C^{2,\alpha}(\overline{B}))$, and $w$ satisfies
$$-\Delta w=\lp \left[2|\uh|^{p-2}\uh-
|\u|^{p-2}\u-
|\u\circ\sigmah|^{p-2}\u\circ\sigmah\right]\geq 0$$
on $H \cap B$. It also satisfies the boundary conditions 
$$
w=0 \quad \text{on }\partial H \cap \overline{B},\qquad  
\frac{\partial w}{\partial\nu}=0\quad {\rm on}\ H\cap\partial B.$$
From Lemma~\ref{sec:appendix-1} we now conclude that either $w\equiv 0$ on $H \cap \overline{B}$ or 
$$ 
w>0 \quad \text{in }\innt(H) \cap \overline{B}, \qquad   
\frac{\partial w}{\partial h}<0 \quad \text{on }\partial H \cap \overline{B}.    
$$ 
In the first case (iii) follows. In the second case, 
we either have
$$
u > u \circ \sigma \quad \text{in }\innt(H) \cap \overline{B}
\qquad \text{and}\qquad  \frac{\partial u}{\partial h}=
\frac{1}{2}\frac{\partial w}{\partial h}<0 \quad \text{on }\partial H \cap \overline{B},
$$
or
$$
u < u \circ \sigma \quad \text{in }\innt(H) \cap \overline{B}
\qquad \text{and}\qquad \frac{\partial u}{\partial h}=
-\frac{1}{2}\frac{\partial w}{\partial h}>0 \quad \text{on }\partial H \cap \overline{B}.
$$
Hence either (i) or (ii) holds. The proof is finished. 
\end{proof}

\begin{altproof}{{\rm \ Proposition~\ref{sec:axial-symm-minim-1}} (completed)} Let $e \in \partial B$ be such that 
$\u(e)= \max\{u(x)\::\: x\in \partial B\}$. Let $\cH_e \subset \cH$ be the set of all half-spaces $H$ in $\Rb^N$ with $0 \in \partial H$ and 
$e \in \innt (H)$. Then Lemma \ref{sec:axial-symm-minim-2} 
and our choice of $e$ imply that $u \ge u \circ \sigmah$ on 
$H \cap \overline{B}$ for every half-space $H \in \cH_e$. 
This however is equivalent to the foliated 
Schwarz symmetry of $\u$ with respect to $e$, as follows 
immediately from \cite[Lemma 2.4]{BWW}, or, alternatively, from \cite[Lemma~4.2]{brock}.

It remains to prove that either $\u$ does not depend on $\theta$, or that $\frac{\partial \u}{\partial \theta}(r,\theta)<0$ for $0<r \le 1,\:0<\theta<\pi$. Obviously, the last property is 
equivalent to
$$
\frac{\partial u}{\partial h}<0\quad \text{on }\partial H \cap \overline{B}\qquad \text{for every }H \in \cH_e. 
$$
We already know that no half-space $H \subset \cH_e$ satisfies property (ii). Now suppose that (iii) applies for some $H_0 \subset \cH_e$. Let $\theta_0$ be the angle between $e$ and the hyperplane $\partial H_0$, which is less than or equal to $\pi/2$. Let 
$e_0= \sigmaho(e)$. Then $\arccos(e_0 \cdot e)=2\theta_0$. Moreover, (iii) implies that 
$u(re_0)=u(re)$ for $0 \le r \le 1$. Since $u$ is nonincreasing in the angle 
$\theta \in (0,\pi)$, we conclude that 
$u(r,\theta)= u(r,0)$ for all 
$\theta \le 2 \theta_0$. From Lemma~\ref{sec:axial-symm-minim-2} we then deduce that (iii) 
holds for all $H_1 \subset \cH_e$ for which the angle between $e$ and $H_1$ is less then $2 \theta_0$. Then, by the same argument as before, $u(r,\theta)= u(r,0)$ for all 
$\theta \le \min\{4 \theta_0,\pi\}$. Arguing successively, in a finite number of steps we obtain $u(r,\theta)= u(r,0)$ for all 
$\theta \le \pi$. This shows that $u$ is radial.
We conclude that either $u$ is a radial function, or 
$$
\frac{\partial u}{\partial h}<0\quad \text{on }\partial H \cap \overline{B}\qquad \text{for every }H \in \cH_e. 
$$
This concludes the proof. 
\end{altproof}

\section{Strict monotonicity in the axial direction}
\label{sec:strict-monot-axial}
In this section we prove parts (b) and (c) of Theorem~\ref{sec:introduction-1}. 
We start with a few preliminaries and recall some known facts. 

\begin{Lem}
\label{sec:strict-monot-axial-2}
Suppose that $u \in C^2(\overline B)$ satisfies 
$\frac{\partial u}{\partial \nu}= 0$ on $\partial B$, where $\nu$ 
is the outward normal. Then
$$
\nabla u(x_0) \cdot \frac{\partial }{\partial \nu}\nabla u(x_0) =
-|\nabla u(x_0)|^2 \qquad \text{for}\quad x_0 \in \partial B.
$$
\end{Lem}

This identity is known, but it seems to be a new ingredient in the present context. It is a special case of an identity used in \cite[Proof of Theorem 2]{casten.holland:1978}. We give a short proof for the convenience of the reader. 

\begin{altproof}{{\rm\ Lemma~\ref{sec:strict-monot-axial-2}}}
Let $r=|x|$, $\theta=\arccos\bigl(\frac{x}{|x|}\cdot e_N\bigr)$, and let $\phi_i$, for $i=1,\ldots,N-2$, be
the other $N-2$ spherical angles for $x \in \overline B$.
We denote by $e_r$, $e_\theta$ and $e_{\phi_i}$
the orthogonal vectors $\frac{\partial}{\partial r}$, $\frac{\partial}{\partial\theta}$ and
$\frac{\partial}{\partial\phi_i}$, respectively.
In this proof we designate by the same letter
functions written in spherical coordinates
and Cartesian coordinates.
Without loss of generality, we suppose $x_0=e_1=(1,0,\ldots,0)$ and
$$\nabla u(x_0)=-\frac{\partial u}{\partial\theta}(x_0)e_N,$$
with $e_N=(0,\ldots,0,1)$.
For $x$ close to $x_0$ and for some functions $u^{\phi_i}$,
$$\nabla u=\frac{1}{r}\frac{\partial u}{\partial\theta}e_\theta+
\frac{\partial u}{\partial r}e_r+\sum_{i=1}^{N-2}u^{\phi_i}e_{\phi_i}.$$
Since by hypothesis $\frac{\partial u}{\partial r}=0$ on $\partial B$, $\frac{\partial^2u}{\partial r\partial\theta}=\frac{\partial^2u}{\partial \theta \partial r}=0$
on $\partial B$.  
Also, $\frac{\partial e_\theta}{\partial r}=\frac{\partial e_r}{\partial r}=
\frac{\partial e_{\phi_i}}{\partial r}=0$. 
Therefore,
$$\left.\nabla u\cdot \frac{\partial }{\partial \nu}\nabla u\right|_{x=x_0} =
\left(-\frac{\partial u}{\partial\theta}(x_0)e_N\right)\cdot
\left.\left(-\frac{1}{r^2}\frac{\partial u}{\partial\theta}e_\theta\right)\right|_{x=x_0}=
-\left(\frac{\partial u}{\partial\theta}(x_0)\right)^2=-|\nabla u(x_0)|^2,$$
as $e_\theta(x_0)=-e_N$.
\end{altproof}

Next we reformulate the minimization problem (\ref{eq:12}) on the unit ball $B$ by introducing a non-homogeneous auxiliary functional. It is convenient to endow the space $\hza=\left\{u\in H^1(B):\textstyle\int_Bu=0\right\}$ with the inner product 
\mbox{$(u,v)_{H^1}=\int_B\nabla u\nabla v$}. We fix 
$p \in (2,2^*]$ for $N\ge 3$, $p>2$ for $N=2$.
We consider the 
$C^2$-functional 
$$
G:\hza\to\Rb,\qquad G(u):=\textstyle\frac{1}{2}\|\nabla u\|_2^2-\frac{1}{p}\|u\|_p^p.
$$
Note that
$$
G'(u)v= \int_B \nabla u \nabla v - \int_B |u|^{p-2}u v \quad \text{and}\quad G''(u)(v,w)= \int_B \nabla v \nabla w - (p-1)
\int_B |u|^{p-2}vw  
$$
for $u,v,w \in \hza$. Thus, a critical point $u \in \hza$ of $G$ is a weak (and therefore $C^{3,\alpha}$) solution of the problem 
\begin{equation}
\label{eq:14}
-\Delta u=|u|^{p-2}u+\mu\ {\rm in}\ B,\quad
\textstyle\frac{\partial u}
{\partial\nu}=0\ {\rm on}\ \partial B.
\end{equation}
with $\mu= -(1/|B|)\int_{B}|u|^{p-2}u$. We consider the {\em Nehari 
manifold}
$$
{\cal N}=\{u\in\hza\setminus\{0\}:\|\nabla u\|_2^2=\|u\|_p^p\}
=\{u\in\hza\setminus\{0\}: G'(u)u = 0 \}.
$$
We recall that $\cN$ is a $C^2$-manifold of codimension one in $\hza$ 
whose tangent space at a point $u \in \cN$ is given by
$$
T_u \cN = \left\{v \in \hza \::\: 2(u,v)_{H^1}-p \int_B |u|^{p-2}u v =0\right\}.
$$
The following lemma is proved by direct computation.

\begin{Lem}
\label{sec:strict-monot-axial-4}
$u \in \hza\setminus\{0\}$ is a 
minimizer of {\rm (\ref{eq:12})} on $B$ if and only if $\bigl(\|\nabla u\|_2^2\,/\,\|u\|_p^p\bigr)^{\frac{1}{p-2}}u \in \cN$ is a minimizer  of the  restriction $G|_\cN$ of $G$ to $\cN$. 
\end{Lem}

So in order to prove parts (b) and (c) of Theorem~\ref{sec:introduction-1},
in the case that $p>2$, it suffices to consider minimizers of 
$G|_{\cN}$. 

\begin{Lem}
\label{sec:strict-monot-axial-3}  
Let $u \in \cN$ be a minimizer of $G|_{\cN}$. Then
\begin{itemize}
\item[{\rm (a)}] $T_u\cN= \{v \in \hza\::\: (u,v)_{H^1}=0\}$. Moreover, $u$ is a critical point of $G$ and hence a solution of {\rm (\ref{eq:14})}.
\item[{\rm (b)}] $G''(u)(v,v) \ge 0$ for $v \in T_u\cN$.
\item[{\rm (c)}] If $v \in T_u\cN \setminus \{0\}$ satisfies $G''(u)(v,v)=0$, then 
$v$ is a solution of
\begin{equation}
\label{eq:38}
-\Delta v= (p-1)\bigl(|u|^{p-2}v+\hat \mu\bigr) \quad \text{in $B$},\qquad 
\frac{\partial v}{\partial \nu}=0 \quad \text{on $\partial B$,}
\end{equation}   
with
$\hat \mu=\hat \mu(u,v)= -(1/|B|)\int_{B}|u|^{p-2}v$.\\ 
If, in addition, $\int_{B}|u|^{p-2}v=0$, then 
$v$ has at most three nodal domains.
\end{itemize}
\end{Lem}

\begin{proof}
(a) Since $u$ is a critical point of $G|_\cN$, we have for $v \in T_u\cN$
$$
0=pG'(u)v=p(u,v)_{H^1}-p\int_B |u|^{p-2}u v= 
(p-2)(u,v)_{H^1}.
$$
Hence $T_u\cN= \{v \in \hza\::\: (u,v)_{H^1}=0\}$. Since furthermore $G'(u)u=0$ by the definition of $\cN$, we conclude that $G'(u)v= 0$ for all $v \in \hza$, and thus $u$ is a critical point of $G$.\\
(b) Let $v \in T_u\cN$, and let $\rho:(-\eps,\eps) \to \cN$ be a $C^2$-curve with $\rho(0)=u$ and $\rho'(0)=v$. Then 
$$
\frac{\partial }{\partial s}G(\rho(s))\Big|_{s=0}=G'(\rho(s))\rho'(s)\Big|_{s=0}=G'(u)v=0
$$
and
\begin{align*}
0 &\le \frac{\partial^2 }{\partial s^2}G(\rho(s))\Big|_{s=0}=
\Bigl(G''(\rho(s))(\rho'(s),\rho'(s))+G'(\rho(s))\rho''(s)\Bigr) \Big|_{s=0}\\
&= G''(u)(v,v)+G'(u)\rho''(0)= G''(u)(v,v),
\end{align*}
since $G'(u)=0$ by (a).\\
(c) Consider the quadratic functional $\phi: T_u\cN \to \R, \:\phi(v)= \frac{1}{2}G''(u)(v,v)$, and consider $v \in T_u\cN$ 
with $G''(u)(v,v)=0$. By (b), $v$ is a minimizer of $\phi$, so that for all $w \in T_u\cN$
$$
0=\phi'(v)w= G''(u)(v,w).
$$
Moreover,
\begin{align*}
G''(u)(v,u)= G''(u)(u,v)&= (u,v)_{H^1} -(p-1)\int_B|u|^{p-2}uv\\
&=(2-p)(u,v)_{H^1}+(p-1)G'(u)v\\
&=(2-p)(u,v)_{H^1}=0.
\end{align*}
We conclude that $G''(u)(v,w)=0$ for all $w \in \hza$, 
hence $v$ is a weak solution of (\ref{eq:38}). By elliptic regularity, $v \in C^{3,\alpha}(\overline B)$ for some $\alpha>0$.
It remains to show that, if \mbox{$\hat \mu(u,v)=0$}, then $v$ has at most three nodal domains. Suppose by contradiction that $v$ has three nodal domains $\Omega_1,\Omega_2,\Omega_3$ such that 
$\{x \in B \::\: v(x) \not= 0\} \setminus (\Omega_1 \cup 
\Omega_2 \cup \Omega_3)$ is a nonempty open set. Let $v_i:= v \chi_{\Omega_i}$, $i=1,2,3$. Then $v_i \in H^1(B)$  
by \cite[Lemma 1]{MP85}. Since $v_1,v_2,v_3$ are linearly independent functions, a suitable nontrivial linear combination 
$\bar v= \sum_{i=1}^3\alpha_i v_i$ satisfies 
$\int_B \bar v=0$ and $(\bar v,\u)_{H^1}= 0$, that is, ${\bar v}\in 
T_{\u}\cN$. Hence $G''(\u)(\bar v,\bar v) \ge 0$ by (b). 
On the other hand, by the disjointness of supports
\begin{align*}
G''(\u)(\bar v,\bar v)&=\int_B |\nabla \bar v|^2 -(p-1)\int_B |u|^{p-2}{\bar v}^2 \\
&= \sum_{i=1}^3 \alpha_i^2\left( \int_B |\nabla v_i|^2 -(p-1)\int_B |u|^{p-2}{v_i}^2 \right) \\
&= \sum_{i=1}^3 \alpha_i^2 \left(\int_B \nabla v \nabla  v_i  -(p-1)\int_B |u|^{p-2}v v_i \right) \\
&= (p-1)\hat \mu(u,v) \sum_{i=1}^3 \alpha_i^2 \int_B v_i = 0,
\end{align*}
so that $\bar v$ is also a minimizer of $\phi$ and hence a solution of (\ref{eq:38}). Since $\bar v \equiv 0$ on the nonempty open set $\{x \in B \::\: v(x) \not= 0\} \setminus (\Omega_1 \cup 
\Omega_2 \cup \Omega_3)$, we conclude that $\bar v$ solves (\ref{eq:38}) with $\hat \mu(u,\bar v)=0$. We now have come to a contradiction to the fact 
that solutions of (\ref{eq:38}) with $\hat \mu=0$ have the weak unique continuation property (see e.g. \cite[p.~519]{simon}).  
\end{proof}

\begin{Prop}
\label{sec:strict-monot-axial-1}
Let $\u \in \cN$ be a minimizer of $G|_\cN$. Then
\begin{itemize}
\item[{\rm (a)}] There exists a unit vector $e \in \R^N,\:|e|=1$ such that $\u(x)$ only 
depends on $r=|x|$ and $\theta:= \arccos\bigl(\frac{x}{|x|}\cdot e\bigr)$, and $\frac{\partial \u}{\partial \theta}(r,\theta)<0$ for 
$0<r \le 1,\:0<\theta<\pi$.  
\item[{\rm (b)}] If $\partial_e$ denotes the directional derivative in the direction of $e$, then \mbox{$\partial_e \u > 0$} 
on $\overline{B} \setminus \{\pm e\}$.
\item[{\rm (c)}] If $\tau$ is another unit vector in $\R^N$ orthogonal to $e$, then $\partial_\tau u$ has precisely
four nodal domains.
\end{itemize}
\end{Prop}

\begin{proof}
(a) Applying Proposition \ref{sec:axial-symm-minim-1} and rotating the coordinate system if necessary, we may assume that $\u$ 
is axially symmetric about the $x_N$-axis and $u(e)\geq u(-e)$, where $e:=e_N=(0,\dots,0,1)$. Moreover, either $u$ is radially symmetric, 
or  $\frac{\partial \u}{\partial \theta}(r,\theta)<0$ for 
$0<r \le 1,\:0<\theta<\pi$, where $r=|x|$ and $\theta$ is the 
angle formed by $\frac{x}{|x|}$ and $e$. We 
consider the partial derivatives 
$\ux \in C^{2,\alpha}({\overline B})$, $i=1,\ldots, N$. 
By differentiating (\ref{eq:14}) we observe that 
\begin{equation}
\label{eq:16}
-\Delta \ux=(p-1)|\u|^{p-2}\ux\quad{\rm in}\ B,\qquad i=1,...,N.
\end{equation}
Moreover, $\ux \in \hza$ and 
$(\ux,\u)_{H^1}=0= \int_B|u|^{p-2}\ux$ for $i=1,\ldots,N-1$,
since $\ux$ is antisymmetric with respect to the hyperplane 
$\{x\in\R^N:\:x_i=0\}$. Hence $\ux \in T_{\u}\cN$ and $G''(\u)(\ux,\ux) \ge 0$ for $i=1,\ldots,N-1$ by Lemma~\ref{sec:strict-monot-axial-3}(a),(b). We 
claim that, for $i=1,...,N-1$, 
\begin{equation}
\label{eq:1}
\text{either}\quad \ux \equiv 0\;\text{ on $B$}\qquad \text{or}\qquad G''(\u)(\ux,\ux) > 0.
\end{equation}
Indeed, suppose by contradiction that $\ux \not\equiv 0$ and
\mbox{$G''(\u)(\ux,\ux)= 0$} for some $i \in \{1,\dots,N-1\}$.
Then, since $\int_B|u|^{p-2}u_{x_i}=0$, Lemma~\ref{sec:strict-monot-axial-3}(c) implies that $\ux$ satisfies (\ref{eq:16}) together with the boundary condition
\begin{equation}
\label{eq:18}
\frac{\partial\ux}{\partial\nu}=0\qquad \text{on}\quad 
\partial B,
\end{equation}
and that $\ux$ has at most three nodal domains. It 
has precisely two nodal domains because it is antisymmetric 
with respect to the hyperplane $\{x\in\R^N:\:x_i=0\}$. 
We may assume that $\ux>0$ in the open 
half-ball $B^i_+:=\{x \in B\::\:x_i>0\}$. 
The homogeneous Neumann boundary condition
for $\u$ implies that $\ux(e_i)=0$, 
where $e_i$ is the $i$-th coordinate vector.  
Since 
$$-\Delta \ux=(p-1)|\u|^{p-2}\ux \:>\:0
\qquad {\rm in}\quad B^i_+,
$$
the Hopf boundary lemma (cf.\ Lemma~\ref{sec:appendix-2} below) forces
$\frac{\partial \ux}{\partial\nu}(e_i)<0$. This 
contradicts~(\ref{eq:18}), and thus (\ref{eq:1}) is proved.
Next we claim that $\u$ is nonradial. Indeed, multiplying the equations (\ref{eq:16}) by $\ux$, respectively, and integrating over $B$, we find
\begin{equation}
\label{eq:15}
G''(\u)(\ux,\ux)=\int_{\partial B}\frac{\partial\ux}{\partial\nu}\ux.
\end{equation}
If we suppose by contradiction that $\u$ is radial, then $\u$ 
is constant on the boundary $\partial B$. Together with the boundary condition $\frac{\partial \u}{\partial\nu}=0$ on $\partial B$ this gives $\nabla\u=0$ on $\partial B$, hence $\ux \equiv 0$ on $B$, $i=1,\dots,N-1$ by (\ref{eq:1}) and (\ref{eq:15}). Then the radial symmetry of $u$ implies that $u$ is constant, which is a contradiction since $u \in \hza \setminus \{0\}$.
Now since $\u$ is nonradial, Proposition~\ref{sec:axial-symm-minim-1} implies that
\begin{equation}
\label{eq:17}
\frac{\partial \u}{\partial \theta}(r,\theta)<0\qquad 
\text{for}\quad 0<r\le 1,\:0<\theta<\pi.
\end{equation} 
We thus have proved (a). 

\noindent (b) The axial symmetry of $\u$ and the Neumann boundary conditions imply
\begin{equation}
\label{eq:101}
\nabla \u(x)= \frac{\partial \u}{\partial \theta}\left(\cos\theta\,\frac{x-(x\cdot e_N)e_N}{|x-(x\cdot e_N)e_N|}-
\sin \theta\, e_N\right)\qquad \text{for}\quad x \in \partial 
B\setminus\{\pm e_N\},
\end{equation}
hence
\begin{equation}
\label{eq:21}
\uxn = -\sin \theta\, \frac{\partial \u}{\partial \theta}\: >\:0
\qquad \text{on}\quad\partial B \setminus \{\pm e_N\}.
\end{equation}
By (\ref{eq:15}) and Lemma~\ref{sec:strict-monot-axial-2} 
we also have 
\begin{equation}
\label{eq:49}
\sum_{i=1}^{N-1} G''(\u)(\ux,\ux)+ \int_{\partial B}\frac{\partial \uxn}{\partial\nu}\uxn = \int_{\partial B}
\frac{\partial(\nabla\u)}{\partial\nu}\cdot\nabla\u
= -\int_{\partial B}|\nabla u|^2 < 0.
\end{equation}
Together with (\ref{eq:1}) this implies 
\begin{equation}
\label{eq:19}
\int_{\partial B}\frac{\partial \uxn}{\partial\nu}\uxn <0.
\end{equation}
We now consider
$$
\Omega_+:= \{ x \in \overline B \::\: \uxn >0\},\qquad 
\Omega_-:= \{ x \in \overline B \::\: \uxn <0\}.
$$
Then $\partial B \setminus \{\pm e_N\} \subset \Omega_+$ by (\ref{eq:21}). We claim that $\Omega_-$ is connected. Indeed, suppose by contradiction that $\Omega_-$ has at least two different 
connected components $\Omega_1$ and $\Omega_2$. 
Let $v_1:= \uxn \chi_{\Omega_1}$ and $v_2:= \uxn \chi_{\Omega_2}$. Then $v_i \in H^1_0(B)$ for $i=1,2$ 
by \cite[Lemma~1]{MP85} and the fact that $\Omega_- \cap \partial B= \varnothing$. We also consider $v_3:= \uxn^+= \uxn \chi_{\Omega_+} \in H^1(B)$. Since the functions $v_1,v_2,v_3$ are linearly independent, a nontrivial linear combination 
$v= \alpha_1 v_1 + \alpha_2 v_2+\alpha_3 v_3$ satisfies 
$\int_B v=0$ and $(v,\u)_{H^1}= 0$, so that $v \in 
T_{\u}\cN$. By Lemma~\ref{sec:strict-monot-axial-3}(b) this implies $G''(\u)(v,v) \ge 0$. On the other hand, by 
the disjointness of supports,
\begin{align*}
G''(\u)(v,v)&=\int_B |\nabla v|^2 -(p-1)\int_B |u|^{p-2}v^2 \\
&= \sum_{i=1}^3 \alpha_i^2 \left(\int_B \nabla \uxn \nabla v_i  -(p-1)\int_B |u|^{p-2}\uxn v_i \right) \\
&= \sum_{i=1}^3 \alpha_i^2 \int_{\partial B}\frac{\partial \uxn}{\partial\nu}v_i= \alpha_3^2 \int_{\partial B}
\frac{\partial \uxn}{\partial\nu}\uxn.
\end{align*}
 Now (\ref{eq:19}) forces $\alpha_3=0$ and $G''(\u)(v,v)=0$. 
 Then, by Lemma~\ref{sec:strict-monot-axial-3}(c), 
 $v= \alpha_1 v_1 + \alpha_2 v_2 \in T_{\u}\cN$ is a solution 
of (\ref{eq:38}). But $v \equiv 0$ on the nonempty open set 
$\Omega_+ \cap B$. This forces $\hat \mu =\hat \mu(u,v)=0$, which contradicts the fact 
that solutions of (\ref{eq:38}) with $\hat \mu=0$ have the weak unique continuation property (see e.g. \cite[p.~519]{simon}). We conclude that $\Omega_-$ is connected. Since
$$
\{x\in\R^N:\:x_N = 0 \} \cap \Omega_- = \varnothing
$$
by (\ref{eq:17}) applied to the angle $\theta=\pi/2$, 
we either have 
\mbox{$\Omega_- \subset B_+:= \{x \in B\::\: x_N >0\}$} or  
\mbox{$\Omega_- \subset B_-:=\{x \in B\::\: x_N <0\}$}. We assume $\Omega_- \subset B_+$ from now on, the other case can be treated similarly. So we already know that 
$\uxn >0$ on $B_-$, and by a moving plane argument we now show 
that $\uxn>0$ on $B_+$.
For $\lambda \ge 0$ we consider the set
$$
B_\lambda= \{x \in {\overline B} \::\: x_N > \lambda \},
$$
whose boundary consists of the sets
$$
S_\lambda:= \{x \in \partial B \::\: x_N >\lambda 
\} \subset B_\lambda \quad \text{and}\quad 
T_\lambda:= \{x \in {\overline B} \::\: x_N =\lambda\}.
$$
We let $v_\lambda \in C^2(\overline{B_\lambda})$ be defined by 
$$
v_\lambda(x',x_N):= \u(x',2\lambda-x_N).
$$
Here $x'=(x_1,\dots,x_{N-1})$. Then the function $w_\lambda:\overline{B_\lambda}\to\Rb$,
defined by 
$w_\lambda:= u - v_\lambda$, satisfies
$$
-\Delta w_\lambda = (|\u|^{p-2}\u -
|v_\lambda|^{p-2}v_\lambda ) = V_{\lambda}(x)w_\lambda
$$
where 
$$
V_{\lambda}(x):= (p-1) \int_0^1|s\u(x) + (1-s)v_\lambda(x)|^{p-2}\:ds \: \ge\: 0, \qquad x \in B_\lambda.
$$
We examine the behavior of $w_\lambda$ on $T_\lambda$.
Let $T^0_\lambda:= \{x \in \partial B \::\: x_N =\lambda\} \subset T_\lambda$.
For $x_0\in T_\lambda^0$ we define 
$$
\frac{\partial w_\lambda}{\partial x_N}(x_0):=\lim_{\stackrel{x\to x_0}{x\in\overline{B_\lambda}\setminus T_\lambda^0}}\frac{\partial w_\lambda}{\partial x_N}(x).
$$
Then on $T_\lambda$ we have $w_\lambda=0$ and $\frac{\partial w_\lambda}{\partial x_N}= 2 \uxn$.
So, (\ref{eq:21}) implies
\begin{equation}
\label{eq:23}
\frac{\partial w_\lambda}{\partial x_N}>0\qquad \text{on}\quad 
T^0_\lambda,
\end{equation} 
for every $\lambda \in [0,1)$.
Next we note that $w_0>0$ in $B_0$ by virtue of (\ref{eq:17}). We denote by $\Lambda$ the biggest interval contained in $[0,1)$ and containing 0, such that  $w_\lambda>0$ in $B_\lambda$ 
for each $\lambda\in\Lambda$. Then 
\begin{equation}
\label{eq:24}
\lambda \in \Lambda \qquad \Lra \qquad \frac{\partial w_\lambda}{\partial x_N}>0\qquad \text{on}\quad 
T_\lambda.
\end{equation}
Indeed, on $T_\lambda \setminus T^0_\lambda$ this follows from the Hopf boundary lemma (see Lemma~\ref{sec:appendix-2} below), while it is a priori true on $T^0_\lambda$ by (\ref{eq:23}). A standard argument based on (\ref{eq:24}) shows that $\Lambda \subset [0,1)$ is relatively open. 
We claim that $\Lambda= [0,1)$.  
Suppose by contradiction that there is 
$0<\lambda<1$ such that $[0,\lambda) \subset \Lambda$ and 
$\lambda \not \in \Lambda$. Then
$$
w_\lambda(x)= \lim_{\kappa\, \nearrow\, \lambda}w_{\kappa}(x) 
\:\ge\: 0 \qquad \text{for}\quad x \in B_\lambda.
$$
Moreover, $\uxn >0$ on $\{x \in B \::\: x_N <\lambda\}$ by 
(\ref{eq:24}) and the preceding arguments. In particular this 
implies $\frac{\partial v_\lambda}{\partial x_N}<0$ on 
$S_\lambda$, whereas $\uxn \ge 0$ on $S_\lambda$ by 
(\ref{eq:21}). Hence    
\begin{equation}
\label{eq:10}
\frac{\partial w_\lambda}{\partial x_N}>0\qquad\text{on}\quad 
S_\lambda.    
\end{equation}
We claim that $w_\lambda>0$ on $\innt(B_\lambda)= B_\lambda \setminus S_\lambda$. Indeed, if $w_\lambda$ had an interior 
minimum point $x_0 \in \innt(B_\lambda)$ with 
$w_\lambda(x_0)=0$, then $w_\lambda \equiv 0$ on $B_\lambda$ 
by the maximum principle. However, by continuity up to the boundary this would yield $\nabla w_\lambda = 0$ on $B_\lambda$, contrary to (\ref{eq:10}). 
Now suppose by contradiction that $w_\lambda(x)=0$ for some 
$x \in S_\lambda$. Then $\frac{\partial w_\lambda}{\partial x_N}(x)\le 0$, and this contradicts (\ref{eq:10}) again. We conclude that $w_\lambda>0$ on $B_\lambda$, and hence $\lambda \in \Lambda$. We arrived at a contradiction.
We have thus proved $\Lambda=[0,1)$, and therefore $\uxn>0$ on $B_+$ by (\ref{eq:24}). 

\noindent (c) To establish the last part of the proposition we assume, without loss of generality, that $\tau=e_1=(1,0,\ldots,0)$.
By (\ref{eq:101}),
$$
u_{x_1} = \frac{x_1}{|x-(x\cdot e_N)e_N|}\cos \theta\, \frac{\partial \u}{\partial \theta}
\qquad \text{on $\partial B \setminus \{\pm e_N\}$}.
$$
Using (\ref{eq:17}), we see $u_{x_1}$ is negative on $\{x\in\partial B:\:x_1>0\ \text{and}\ x_N>0\}$ and on
the reflection of this set at the origin, and we see $u_{x_1}$ is positive
on $\{x\in\partial B:\:x_1>0\ \text{and}\ x_N<0\}$ and on
the reflection of this set at the origin.  Also, $u_{x_1}=0$ on $x_1=0$.
So, the function $u_{x_1}$ has exactly four nodal domains touching $\partial B$.
Suppose $u_{x_1}$ has more than four nodal domains.  Then we can choose $\Omega_1$,
a nodal domain for $u_{x_1}$ not intersecting $\partial B$, with say $u_{x_1}>0$
on $\Omega_1$.  Because $u_{x_1}$ is antisymmetric with respect to the
hyperplane $H_{e_1}:=\{x\in\R^N:\: x_1=0\}$, the reflection, $\Omega_2$, of $\Omega_1$ at
$H_{e_1}$ is also a nodal domain for $u_{x_1}$, but with
$u_{x_1}<0$ on $\Omega_2$.  Let $v_0:=u_{x_1}\chi_{\Omega_1}+u_{x_1}\chi_{\Omega_2}$.
Again, $v_0\in H^1_0(B)$ by \cite[Lemma 1]{MP85}. Moreover, $v_0\in\hza$ and
$(v_0,u)_{H^1}=0$, due to the antisymmetry of $v_0$ and the symmetry of $u$,
with respect to $H_{e_1}$. 
By Lemma~\ref{sec:strict-monot-axial-3}(a), $v_0\in T_u{\cal N}$.
Multiplying the equation (\ref{eq:16}), for $i=1$,
by $v_0$ and integrating over $B$ yields
$$
G''(u)(v_0,v_0)=0.
$$
Lemma~\ref{sec:strict-monot-axial-3}(c) implies that $v_0$ is a solution of (\ref{eq:38}).
Here we note that $\hat \mu=\hat \mu(u,v_0)=0$, once more because of the symmetry
properties of $u$ and $v_0$.
Applying the unique continuation principle, we arrive at a 
contradiction.  Therefore, all nodal domains of $u_{x_1}$ intersect $\partial B$,
and $u_{x_1}$ has precisely four nodal domains.
The proof of Proposition~\ref{sec:strict-monot-axial-1} is complete.
 
\end{proof}

\section{Antisymmetry and uniqueness for $p$ close to 2}
\label{sec:antisymmetry-p-close}
In this section we prove part (d) of Theorem~\ref{sec:introduction-1}. We also prove that, unlike for the case $p=2$, for $p>2$ close to 2
there do not exist functions $R:[0,1]\to\R$
and $\Theta:[0,\pi]\to\R$ such that $u(r,\theta)=R(r)\Theta(\theta)$.
We define the space
$$H_z:=\left\{u\in\hza\::\: u(Ax',x_N)=u(x',x_N)\; \text{for all}\:A \in O(N-1)\right\},$$ 
where $x'\in\Rb^{N-1}$, $x_N\in\Rb$ and $(x',x_N)\in B$.
We proved in Section \ref{sec:axial-symm-minim} that, modulo a rotation, every minimizer of (\ref{eq:12}) in $\hza$ belongs to $H_z$, 
so that
\begin{equation}
\label{eq:30}
\Lp\,=\, \inf_{\hza\setminus\{0\}}\frac{\|\nabla u\|_2^2}{\|u\|_p^2}\,=\, \inf_{H_z\setminus\{0\}}\frac{\|\nabla u\|_2^2}{\|u\|_p^2}.
\end{equation}
We also consider the subspace $\has \subset H_z$ of the functions in $H_z$ which are antisymmetric with respect to the plane
$\{x\in\R^N\::\:x_N=0\}$,
\begin{equation}
\label{eq:170}
\has:=\{u\in H_z\:: \: u(x',-x_N)=-u(x',x_N)\}.
\end{equation}
Then
$$
\Lp \leq\Lpp:= \inf_{\has\setminus\{0\}}\frac{\|\nabla u\|_2^2}{\|u\|_p^2}.
$$
It is easy to see that the values $\Lp$ and $\Lpp$ depend continuously on $p \in [2,2^*)$. Moreover, $\Lambda_2=\Lambda_2'$.
Indeed, 
by the discussion in Section~\ref{sec:case-p=2}, the intersection of 
the eigenspace corresponding to $\Lambda_2$ with $H_z$ is a one-dimensional 
subspace of $H_z$, and the minimum $\Lambda_2$ is achieved by an (up to a constant factor) unique eigenfunction $u_2$ which belongs to $\has$. Hence $\Lambda_2=\|\nabla u_2\|_2^2/\|u_2\|_2^2 = \Lambda_2'$. Now Theorem~\ref{sec:introduction-1}(d) can be rephrased in the 
following way.  

\begin{Prop}
\label{sec:antisymmetry-p-close-1}
For $p>2$ close to $2$, $\Lp = \Lpp$, the minimizer $u \in H_z$ of\/ {\rm (\ref{eq:30})} 
is unique (up to multiplication by a constant), and it belongs to $\has$.
\end{Prop}

\begin{proof} For $2 \le p<2^*$, let $\up\in H_z$ be such that
$\|\nabla \up\|_2=1$ and $\|\up\|_p^{-2}=\Lp$. Let $\vp \in H_z$ be defined by 
$v_p(x',x_N) =-u_p(x',-x_N)$ for 
$(x',x_N) \in B$. 
Then also $\|\nabla \vp\|_2=1$ and $\|\vp\|_p^{-2}=\Lp$. Hence both $\up$ and $\vp$ are solutions of (\ref{eq:37}) with $\lp= \Lp^{p/2}$ and $\mu_p=\mu_p(\up)=\mu_p(\vp)$. By the remarks above, $u_2=v_2$ is an eigenfunction of the Neumann Laplacian on $B$ corresponding to the first nontrivial eigenvalue $\lambda_2$. We claim that 
\begin{equation}
\label{eq:3}
\up= v_p \qquad \text{for $p>2$ close to $2$.}
\end{equation}
Arguing by contradiction, we suppose there exists a sequence of numbers $p_n>2$, $p_n \to 2$ as $n\to\infty$ 
such that $u_{p_n}\not = v_{p_n}$. For ease of notation we will omit the index $n$. From standard elliptic estimates, we deduce that the 
sequences $(\up)$ and $(\vp)$ are {\em uniformly} bounded in 
$C^{2,\alpha}(\overline{B})$ for some positive $\alpha$. 
Hence, by compactness of the embedding $C^{2,\alpha}(\overline{B}) \subset C^2(\overline{B})$, we may pass to subsequences of $(\up)$ and $(\vp)$ which converge in $C^2(\overline{B})$, respectively. In fact, since 
$$
\frac{\|\nabla \up\|_2}{\|\up\|_p^2}= \frac{\|\nabla \vp\|_2}{\|\vp\|_p^2} \to \Lambda_2 \qquad \text{as $p \to 2$,}
$$
the remarks before 
Proposition~\ref{sec:antisymmetry-p-close-1} imply that, after changing signs if necessary,  
$\up \to u_2$ and $\vp \to u_2$ in $C^2(\overline{B})$, where $u_2 \in \has$ is as above. We now put $\wpp:=\up-\vp \in H_z$ and  
$$
\wt:=\frac{\wpp}{\|\nabla\wpp\|_2}.
$$
We can assume that, as $p\to 2$, $\wt$ converges weakly to 
some $\tilde{w}$ in $H_z$, hence
\begin{equation}
\label{eq:32}
\wt \to \tilde{w} \qquad \text{strongly in $L^q(B)$ for $q<2^*$.}
\end{equation}
We want to derive an equation for 
$\tilde{w}$. The functions $\wpp$ satisfy
\begin{equation}
\label{eq:6}
-\Delta\wpp=\lp (|\up|^{p-2}\up-|\vp|^{p-2}\vp)=\lp V_p \wpp
\quad {\rm in}\ B,\qquad \textstyle\frac{\partial \wpp}
{\partial\nu}= 0 \quad {\rm on}\ \partial B,
\end{equation}
where $V_p:B\to\Rb$ is defined by
$$
V_p(x):=(p-1)\int_0^1|s\up(x)+(1-s)\vp(x)|^{p-2}\,ds.
$$
Also, the functions $\wt$ satisfy 
\begin{equation}
\label{eq:36}
-\Delta\wt=\lp V_p \wt \quad {\rm in}\ B,\qquad \textstyle\frac{\partial \wt}
{\partial\nu}= 0 \quad {\rm on}\ \partial B.
\end{equation}
We claim that 
\begin{equation}
\label{eq:7}
\lim_{p \to 2} \|1-V_p\|_q \to 0 \qquad \text{for every}\quad 
q< \infty. 
\end{equation}
Indeed, note that $u_2(x) \not= 0$ for $x \in  B$ with $x_N \not= 0$, and for these $x$ we have
$$
V_p(x)=(p-1)\int_0^1|s\up(x)+(1-s)\vp(x)|^{p-2}\,ds \to 1 \qquad \text{as }p \to 2,
$$
since  
$\lim \limits_{p \to 2}|s\up(x)+(1-s)\vp(x)|=|u_2(x)|>0$ uniformly in $s \in [0,1]$. Moreover, using
$V_p\geq 0$,
\begin{align*}
|1-V_p(x)|^q &\le 1+|V_p(x)|^q \le 1+ \left[(p-1)
\int_0^1(s|\up(x)|+(1-s)|\vp(x)|)^{p-2}\,ds\right]^q \\
&\le 1+\bigl[(p-1)(|\up(x)|+|\vp(x)|)^{p-2}\bigr]^q,\\
&\le c  \quad \text{in }B
\end{align*}
with a constant $c>0$, since $\up$ and $\vp$ are uniformly bounded on $B$. Hence (\ref{eq:7}) follows from Lebesgue's theorem.\\
Taking the limit as $p \to 2$ in (\ref{eq:36}) and using (\ref{eq:7}), we find that $\tilde{w}$ is a weak solution of the problem
$$
-\Delta\tilde{w}=\lambda_2\tilde{w}\quad {\rm in}\ B,\qquad
\textstyle\frac{\partial\tilde{w}}
{\partial\nu}=0\quad {\rm on}\ \partial B,
$$
Using (\ref{eq:32}), (\ref{eq:36}) and (\ref{eq:7}) we now get 
$$
\|\nabla\tilde{w}\|_2^2=\lambda_2\|\tilde{w}\|_2^2=
\lim_{p\to 2} \lp\int_B V_p\tilde{w}_p^2 =\lim_{p\to 2}\|\nabla\wt\|_2^2=1,
$$
so that $\tilde w_p \to \tilde w$ strongly in $H_z$. Hence $\tilde w \in H_z$ is a normalized eigenfunction of the Neumann Laplacian on $B$ corresponding 
to the eigenvalue $\lambda_2$. By the remarks before Proposition \ref{sec:antisymmetry-p-close-1}, we conclude that 
$\tilde w=\pm u_2$. However, since $\tilde w_p \to \tilde w$, $u_p \to u_2$ and $v_p \to v_2$ in $H_z$, 
we also get
\begin{align*}
\int_B\nabla\tilde{w}\nabla u_2&= \lim_{p\to 2}\int_B\nabla\wt\nabla\up =\lim_{p\to 2}\frac{\|\nabla \up\|_2^2-\int_B\nabla\vp\nabla\up}{\|\nabla(\up-\vp)\|_2}\\
&=\lim_{p\to 2}\frac{1-\int_B\nabla\vp\nabla\up}{\Bigl(2-
2\int_B\nabla\vp\nabla\up \Bigr)^{1/2}}
=\frac{1}{\sqrt{2}}\lim_{p\to 2}
\Bigl(1-\int_B\nabla\vp\nabla\up\Bigr)^{1/2}=\:0.
\end{align*}
This contradiction shows (\ref{eq:3}), which means that $\up \in \has$ for $p>2$ close to $2$. In particular, this shows $\Lp = \Lpp$ for $p>2$ close to $2$.\\
It remains to prove uniqueness (up to a constant) of minimizers {\em in $\has$} for $p>2$ close to $2$. So now suppose by contradiction that, for a sequence of numbers $p_n>2$, $p_n \to 2$ as $n\to\infty$ 
there exists $u_{p_n}, v_{p_n} \in \has$ such that 
$\|\nabla u_{p_n}\|_2=\|\nabla v_{p_n}\|_2=1$, 
$\|u_{p_n}\|_p^{-2}=\|v_{p_n}\|_p^{-2}=\Lp$ and $u_{p_n} \not = 
\pm v_{p_n}$ for all $n$. Passing to a subsequence and changing signs if 
necessary, we may assume that $u_{p_n} \not = v_{p_n}$ for all $n$, and that $u_{p_n},v_{p_n} \to u_2$ in $C^2(\overline{B})$ as $n \to \infty$. Omitting again the index $n$,
we note that, by antisymmetry,$$
-\Delta\up=\lp\,|\up|^{p-2}\up\quad {\rm in}\ B,\qquad
\textstyle\frac{\partial\up}
{\partial\nu}=0\quad {\rm on}\ \partial B,
$$
and
$$
-\Delta\vp=\lp\,|\vp|^{p-2}\vp\quad {\rm in}\ B,\qquad
\textstyle\frac{\partial\vp}
{\partial\nu}=0\quad {\rm on}\ \partial B,
$$
with $\lp= \Lp^{p/2}$. 
To reach a contradiction one argues precisely
as before, considering the normalized difference 
$\wt:=\frac{\up-\vp}{\|\nabla(\up-\vp)\|_2}$.
The proof of Proposition~\ref{sec:antisymmetry-p-close-1} is complete.
\end{proof}

A variant of the argument in the proof of Proposition~\ref{sec:antisymmetry-p-close-1}
shows the following result. We omit the details. 

\begin{Thm}
\label{sec:antisymmetry-p-close-2}
Let $\Omega \subset \R^N$ be a smooth, bounded domain which 
is symmetric with respect to some hyperplane $H$, and such that the first nontrivial eigenvalue $\lambda_2(\Omega)$ of the Neumann Laplacian on $\Omega$ is simple. 
Then
\begin{itemize}
\item[{\rm (a)}] If the (up to a constant) unique eigenfunction $u_2$ of $-\Delta$ corresponding to $\lambda_2(\Omega)$ is symmetric with respect to the reflection at $H$, then, for $p>2$ close to $2$, the minimizer $\up$ of {\rm (\ref{eq:12})} is 
unique (up to a constant) and symmetric with respect to the reflection at $H$.  
\item[{\rm (b)}] If the (up to a constant) unique eigenfunction $u_2$ of $-\Delta$ corresponding to $\lambda_2(\Omega)$ is antisymmetric with respect to the reflection at $H$, then, for $p>2$ close to $2$, the minimizer $\up$ of {\rm (\ref{eq:12})} is
unique (up to a constant) and antisymmetric with respect to the reflection at $H$.  
\end{itemize}
\end{Thm}
For $N=2$, the assumption that $\lambda_2(\Omega)$ 
is simple can often be deduced from geometrical properties of $\Omega$, see
\cite[Section 2]{BB} and the references therein.  

We end this section with two remarks.

\begin{Rem}
\label{sec:strict-monot-axial-5}
Suppose $u$ is as in {\rm Theorem~\ref{sec:introduction-1}}, with $u$ antisymmetric.
There do not exist functions $R:[0,1]\to\R$
and $\Theta:[0,\pi]\to\R$ such that $u(r,\theta)=R(r)\Theta(\theta)$.
\end{Rem}
\begin{proof}
Without loss of generality we may assume $\|\nabla u\|_2=1$.
Since $u$ depends only on $r$ and $\theta$, the Laplacian of $u$ in polar coordinates writes as 
\begin{equation}
\label{eq:106}
\Delta u=\frac{1}{r^{N-1}}\frac{\partial}{\partial r}\bigl(u_rr^{N-1}\bigr)+
\frac{1}{r^2\sin^{N-2}\theta}\frac{\partial}{\partial\theta}\bigl(u_\theta \sin^{N-2}\theta \bigr).
\end{equation}
Now, the function $u$ satisfies (\ref{eq:37}) with $\mu_p=0$.  
Let us assume, by contradiction,
that there exist functions $R:[0,1]\to\R$
and $\Theta:[0,\pi]\to\R$ such that $u(r,\theta)=R(r)\Theta(\theta)$.  
Then, since $u\in C^{3,\alpha}(\overline{B})$, $R$ and $\Theta$ are $C^2$-functions.
Substituting this ansatz for $u$
into (\ref{eq:37}) and using (\ref{eq:106}), we obtain
$$
-\frac{1}{r^{N-1}}\frac{\partial}{\partial r}\bigl(R_rr^{N-1}\bigr)\Theta-
\frac{1}{r^2\sin^{N-2}\theta}\frac{\partial}{\partial\theta}\bigl(\Theta_\theta \sin^{N-2}\theta \bigr)R=
\lambda_p|R|^{p-2}|\Theta|^{p-2}R\Theta
$$
for $r\neq 0$, or
\begin{equation}
\label{eq:107}
\lambda_p|R|^{p-2}r^2|\Theta|^{p-2}=
-\frac{1}{R\,r^{N-3}}\frac{\partial}{\partial r}\bigl(R_rr^{N-1}\bigr)-
\frac{1}{\Theta\sin^{N-2}\theta}\frac{\partial}{\partial\theta}\bigl(\Theta_\theta \sin^{N-2}\theta \bigr)
=:a(r)+b(\theta)
\end{equation}
for $r$, $R(r)$ and $\Theta(\theta)\ne 0$.
Fix two values $0\leq\theta_1<\theta_2\leq\pi$, 
such that $0\neq|\Theta(\theta_1)|\neq|\Theta(\theta_2)|\neq 0$. This is possible 
by (\ref{eq:17}). Subtracting (\ref{eq:107}) evaluated at $\theta_1$ and (\ref{eq:107}) evaluated at $\theta_2$,
\begin{equation}
\label{eq:108}
\lambda_p|R(r)|^{p-2}r^2\bigl(|\Theta(\theta_1)|^{p-2}-|\Theta(\theta_2)|^{p-2}\bigr)=b(\theta_1)-b(\theta_2),
\end{equation}
for $r\neq 0$ such that $R(r)\neq 0$.  For every such $r$ we read out from (\ref{eq:108}) that
$$|R(r)|^{p-2}r^2=c^{p-2},$$
with $c$ a fixed constant, or
\begin{equation}
\label{eq:109}
R(r)=\frac{c}{r^{2/(p-2)}}.
\end{equation}
Now, there must exist some $r\in\,(0,1]$ such that $R(r)\neq 0$,
otherwise $u\equiv 0$. Pick such an $r$. If we use (\ref{eq:109})
and the continuity of $u$ on $\overline{B}$, and thus of $R$ on $[0,1]$, we conclude the function $R$
never vanishes for $r\neq 0$ and $\lim_{r\to 0}R(r)=\infty$. This is impossible so
we have reached a contradiction. Hence, it is not true that there exist functions $R:[0,1]\to\R$
and $\Theta:[0,\pi]\to\R$ such that $u(r,\theta)=R(r)\Theta(\theta)$. Remark~\ref{sec:strict-monot-axial-5}
is proved.
\end{proof}
\begin{Rem}
\label{sec:strict-monot-axial-g}
Suppose $u$ is as in {\rm Theorem~\ref{sec:introduction-1}}, with $u$ antisymmetric.
There do not exist functions $R:[0,1]\to\R$
and $Z:[-1,1]\to\R$ such that $u(x)=R(\rho)Z(x\cdot e)$, with $\rho=|x-(x\cdot e)e|$. 
\end{Rem}
The proof is similar.  We omit the details.

\section{Antisymmetry breaking for large $p$ in the two dimensional case}
\label{sec:antisymm-break-large}
In this section, we consider a situation where antisymmetry fails for the extremal functions. Recall the 
definitions 
$$
\Lp= \inf_{u \in \hza \setminus \{0\}}\frac{\|\nabla u\|_2^2}{\|u\|_p^2},\qquad \Lpp= \inf_{u \in \has \setminus \{0\}}\frac{\|\nabla u\|_2^2}{\|u\|_p^2},
$$
where $\has$ was introduced in (\ref{eq:170}). We restrict our attention to the case
$N=2$, since the following arguments only apply 
in this case. We wish to prove the following.

\begin{Prop}
\label{sec:antisymm-break-large-3}
There exists $p_0>2$ such that 
$\Lp<\Lpp$ for $p>p_0$. Hence the minimizers of 
{\rm (\ref{eq:12})} on $B$ are not antisymmetric for $p>p_0$. 
\end{Prop}

We start the proof of this proposition 
by considering an {\em arbitrary} domain $\Omega\subset\Rb^2$, and we put
$$
\hat{\Lambda}_p(\Omega):=\inf_{u\in H^1_0(\Omega)\setminus \{0\}}
\frac{\|\nabla u\|_2^2}{\|u\|_p^2}.
$$
We quote the following from \cite[Lemma 2.2]{Ren-Wei}.

\begin{Lem}
\label{sec:antisymm-break-large-1}
For any smooth bounded domain $\Omega\subset\Rb^2$, 
$\lim \limits_{p \to \infty}p \hat{\Lambda}_p(\Omega)= 8 \pi e$.
In particular, $\hat{\Lambda}_p(\Omega) \to 0$ as 
$p \to \infty$. 
\end{Lem}

\begin{Cor}
\label{sec:antisymm-break-large-2}
$\Lpp\to 0$ as $p\to\infty$.
\end{Cor}

\begin{proof}
Let $B_+= \{x \in B\::\: x_N >0\}$, and let 
$u \in H^1_0(B_+) \setminus \{0\}$ be a function with 
$\frac{\|\nabla u\|_2^2}{\|u\|_p^2}= \hat{\Lambda}_p(B_+)$.
Then the function $w \in \has$ defined by 
$$
w(x)= \left \{
  \begin{aligned}
   &u(x), &&\qquad x \in B_+,\\
   -&u(x_1,\dots,x_{N-1},-x_N), &&\qquad x \in B \setminus B_+,  
  \end{aligned}
\right.
$$
satisfies
$$
\frac{\|\nabla w\|_2^2}{\|w\|_p^2}= 
2^{1-\frac{2}{p}}\frac{\|\nabla u\|^2_2}{\|u\|_p^2}
=2^{1-\frac{2}{p}}\hat{\Lambda}_p(B_+). 
$$
By Lemma~\ref{sec:antisymm-break-large-1} we conclude 
that $\Lpp \le 2^{1-\frac{2}{p}}\hat{\Lambda}_p(B_+) \to 0$ as 
$p \to \infty$, as claimed.
\end{proof}

\begin{altproof}{{\rm \ Proposition~\ref{sec:antisymm-break-large-3}} (completed)}
For every $p$, let $\vp\in\has$, with $\|\nabla v_p\|_2=1$,
be such that $\Lpp=\|v_p\|_p^{-2}$. Since $\vp=0$ on
$\{x\in\R^N\::\:x_N=0\}\cap B$, we can define $\bar{u}_p\in H^1(B)$ by 
setting $\bar{u}_p(x)=v_p(x)$ for $x \in B_+$ and $\bar{u}_p(x)=0$ for 
$x \in B \setminus B_+$. Note that 
$$
\|\nabla \bar{u}_p\|_2^2=\textstyle{\frac{1}{2}}\|\nabla\vp\|_2^2
=\textstyle{\frac{1}{2}}\quad{\rm and}\quad
\|\bar{u}_p\|_p=\left(\textstyle{\frac{1}{2}}\right)^{1/p}\|v_p\|_p=\left(\textstyle{\frac{1}{2}}\right)^{1/p}(\Lpp)^{-1/2} 
.
$$
From Poincaré's inequality and the Sobolev embedding, there exists a constant 
$C>0$, independent of $p$, such that
$$
\|\bar{u}_p\|_1\leq C.
$$
Consider $\tilde{u}_p\in
H_{za} \setminus \{0\}$ defined by $\tilde{u}_p=\bar{u}_p-(1/|B|)\int_B\bar{u}_p$. Then
\begin{align*}
\Lp &\leq  \frac{\|\nabla \tilde{u}_p\|_2^2}{\|\tilde{u}_p\|_p^2} =\frac{\|\nabla\bar{u}_p\|_2^2}{\|\bar{u}_p-(1/|B|)
 \int_B \bar{u}_p\|_p^2}
\ \leq\ {\frac{1}{2}}\biggl(\|\bar{u}_p\|_p-\Bigl \|\frac{1}{|B|}\int_B\bar{u}_p\Bigr \|_p\biggr)^{-2}\\
&\leq  {\frac{1}{2}}\biggl[\left(\textstyle{\frac{1}{2}}\right)^{1/p}(\Lpp)^{-1/2}-\frac{C}{|B|^{1-1/p}}\biggr]^{-2}\\
&=\ \frac{1}{2^{1-2/p}}\biggl[1-(\Lpp)^{1/2}\frac{2^{1/p}C}{|B|^{1-1/p}}\biggr]^{-2}\:\Lpp
\:<\: \Lpp
\end{align*}
for $p$ sufficiently large, since $\Lpp \to 0$ as 
$p\to+\infty$ by Corollary~\ref{sec:antisymm-break-large-2}.
We have completed the proof of Proposition~\ref{sec:antisymm-break-large-3}.
\end{altproof}
\noindent Theorem~\ref{sec:introduction-1}(e) is proved.
We remark that a variant of the argument given above 
yields the following result. We omit the details.

\begin{Thm}
\label{sec:antisymm-break-large-4}
Let $\Omega \subset \R^2$ be a smooth, bounded domain which is symmetric with respect to some hyperplane $H$. Then, for large $p$, the minimizers of {\rm (\ref{eq:12})} are not antisymmetric with respect to the reflection at $H$. 
\end{Thm}

\section{The case of an annulus}
\label{sec:annulus}

In this section we briefly discuss the case where $\Omega=\A=\{x\in\R^N:\:\rho<|x|<1\}$, for some 
fixed $0<\rho<1$. Suppose $2\leq p\leq 2^*$ if $N\geq 3$,
or $2\leq p<\infty$ if $N=2$. Let $u$ be a minimizer for {\rm (\ref{eq:12})} on $\A$.
Then there exists a unit vector
$e\in \R^N,\:|e|=1$ such that $u(x)$ only 
depends on $r=|x|$ and $\theta= \arccos\bigl(\frac{x}{|x|}\cdot e\bigr)$, and
\begin{equation}
\label{eq:20}
\frac{\partial u}{\partial \theta}(r,\theta)<0\qquad\text{for}\quad \rho\leq r \le 1,\ 0<\theta<\pi.
\end{equation}
This follows by similar arguments as in the case of the ball, 
see Section~\ref{sec:axial-symm-minim}, Lemmas~\ref{sec:strict-monot-axial-2}--\ref{sec:strict-monot-axial-3} and Proposition~\ref{sec:strict-monot-axial-1}(a). One just has to use 
Remark~\ref{sec:appendix-3} instead of Lemma~\ref{sec:appendix-1}. 

If $p\ge 2$ is close to 2, then $u$ is antisymmetric with respect to reflection at $\{x\in\R^N:\:x \cdot e =0\}$, and all other minimizers of (\ref{eq:12}) on $\A$ having the same symmety axis as $u$
are multiples of $u$. This is proved as in the case of the ball, see Proposition~\ref{sec:antisymmetry-p-close-1}.

Henceforth we suppose $e=e_N$, and we discuss the sign of the derivative $\partial_e u= \frac{\partial u}{\partial x_N}$.   
The Neumann boundary conditions and (\ref{eq:20}) imply
$$
\nabla u(x)= \frac{\partial u}{\partial \theta}\left(\cos\theta\,\frac{x-(x\cdot e_N)e_N}{|x-(x\cdot e_N)e_N|}
-\sin \theta\, e_N\right)\qquad \text{for}\quad x \in \partial 
\A\setminus\{\pm\rho  e_N,\pm e_N\},
$$
so that  
\begin{equation}
\label{eq:201}
\upxn = -\sin \theta\, \frac{\partial u}{\partial \theta}\: >\:0
\qquad \text{on}\quad \partial \A\setminus\{\pm\rho  e_N,\pm e_N\}.
\end{equation}
The method we used to show that $\frac{\partial u}{\partial x_N}>0$ for the ball (see the proof of 
Proposition~\ref{sec:strict-monot-axial-1}(b)) does not carry over to the annulus. However, in the special case $p=2$, this property can be verified by a direct computation similar as in Section \ref{sec:case-p=2}.  We now consider the set of values $q$ such that for each $p\in[2,q)$ the minimizer of (\ref{eq:12}) on $\A$ with $e=e_N$ is
unique (up to multiplication by a constant).
Let $p_N$ be supremum of this set. From the above remarks, we know $p_N>2$. Moreover, in dimension $N=2$, Theorem~\ref{sec:antisymm-break-large-4} yields $p_2< \infty$.   

\begin{Prop}
\label{sec:annulus-1} Suppose $p_N$ is as above.
For $2\leq p<p_N$ denote by $u_p$ the unique minimizer for {\rm (\ref{eq:12})} on $\A$,
axially symmetric with respect to the axis passing through zero and $e$, with
$\|\nabla u\|_2=1$ and $u(e)>u(-e)$. Then
$\partial_e u_p>0$ on $\overline \A \setminus\{\pm\rho  e_N,\pm e_N\}$.
\end{Prop}

\vspace{1mm}

\noindent{\bf Open Problem:} Is $\partial_e u>0$ on $\overline{\A} \setminus\{\pm\rho  e_N,\pm e_N\}$ for $p\geq p_N$?

\vspace{2mm}

\begin{altproof}{{\rm Proposition~\ref{sec:annulus-1}}}
Consider the assertion
\begin{equation}
\label{eq:210}
\partial_e u_p \geq 0\qquad \text{on}\ \A.
\end{equation}
By the above remarks, (\ref{eq:210}) is true for $p=2$. Let $p_0\geq 2$. First we show
\begin{equation}
\label{eq:61}
\left\{\begin{array}{l}
p_0<p_N\\
\text{(\ref{eq:210}) is true in $[2,p_0]$}
\end{array}\right.
\qquad\Rightarrow\qquad
\begin{array}{l}
\text{there exists $\delta>0$ such that}\\
\text{(\ref{eq:210}) is true in $[2,p_0+\delta)$.}
\end{array}
\end{equation}
Then we show 
\begin{equation}
\label{eq:62}
\left\{\begin{array}{l}
p_0<p_N\\
\text{(\ref{eq:210}) is true in $[2,p_0)$}
\end{array}\right.
\qquad\Rightarrow\qquad
\begin{array}{l}
\text{(\ref{eq:210}) is true in $[2,p_0]$.}
\end{array}
\end{equation}
Statements (\ref{eq:61}) and (\ref{eq:62}) together imply
(\ref{eq:210}) is true in $[2,p_N)$.

\noindent (a) Suppose $2\leq p_0<p_N$ and (\ref{eq:210}) is true in $[2,p_0]$. 
We define
$$
\opm:= \{ x \in \overline \A \::\: \upxn <0\}.
$$
By assumption $(\Omega_{u_{p_0}})_{-}$ is empty.
We will now show that $\opm=\emptyset$ for $p>p_0$ close to $p_0$.  
Suppose, by contradiction, that $(\Omega_{u_{p_n}})_- \neq\emptyset$ for a sequence $p_n \searrow p_0$. As before, we omit the index $n$. Then we may choose $z_p\in(\Omega_{u_p})_{-}$.
Modulo a subsequence, $z_p\to z_0$ and the sequence $(u_p)$ 
converges to $u_{p_0}$ in $C^2(\overline{\A})$.
We know $(u_{p_0})_{x_N}\geq 0$ by assumption and we have $(u_{p_0})_{x_N}(z_0)=0$.  
We consider three cases:
\begin{enumerate}
\item[{\rm (i)}] $z_0\in\partial\A\setminus\{\pm\rho e_N,\pm e_N\}$.  Then $(u_{p_0})_{x_N}(z_0)=0$ contradicts (\ref{eq:201}).
\item[{\rm (ii)}] $z_0\in\A$.  The function $(u_{p_0})_{x_N}$ satisfies 
\begin{equation}
\label{eq:205}
-\Delta (u_{p_0})_{x_N}=(p_0-1)\lambda_{p_0}|u_{p_0}|^{p_0-2}(u_{p_0})_{x_N}\geq 0\quad{\rm in}\ \A,
\end{equation}
with $\lambda_{p_0}=[\Lambda_{p_0}(\A)]^{p_0/2}$. Then $(u_{p_0})_{x_N}(z_0)=0$ contradicts the strong
maximum principle.
\item[{\rm (iii)}] $z_0\in\{\pm\rho e_N,\pm e_N\}$.  We suppose that $z_0=e_N$, the other cases being treated similarly.
Applying the Mean Value Theorem to $(u_p)_{x_N}$, there exists $\bar{z}_p\in\R^N$ in 
the line through $z_p$ parallel to the $x_N$-axis, with the $N$-th coordinate $(\bar{z}_P)_N>(z_P)_N$,
such that $(u_p)_{x_Nx_N}(\bar{z}_p)>0$. The $C^2$ convergence of $u_p$ to $u_{p_0}$ implies
$(u_{p_0})_{x_Nx_N}(z_0)\geq 0$. Recalling that $(u_{p_0})_{x_N}$ satisfies (\ref{eq:205}), $(u_{p_0})_{x_N}\not\equiv 0$ and
$(u_{p_0})_{x_N}(z_0)=0$, Hopf's lemma implies $(u_{p_0})_{x_Nx_N}(z_0)<0$.  This is a contradiction.
\end{enumerate}
Since all three cases are impossible, 
we conclude that $\opm=\emptyset$ for $p$ close to $p_0$. This establishes (\ref{eq:61}).

\noindent (b) Suppose $p_0<p_N$ and (\ref{eq:210}) is true in
$[2,p_0)$. Taking an increasing sequence $p\to p_0$,
$(u_p)$ converges to $u_{p_0}$ in $C^2(\overline{\A})$, which can easily be deduced from the uniqueness of $u_{p_0}$. Thus (\ref{eq:210}) is true for $p=p_0$.
This establishes (\ref{eq:62}).

So we know (\ref{eq:210}) is true in $[2,p_N)$.
This, (\ref{eq:201}) and
the strong maximum principle imply the assertion.
\end{altproof}

\section{Appendix}
\label{sec:appendix}

We recall the classical Hopf boundary lemma, see e.g. \cite[Lemma 3.4]{GT}.

\begin{Lem}
\label{sec:appendix-2}
Let $\Omega \subset \R^N$ be a domain and $x_0 \in \partial \Omega$ be a boundary point where the interior sphere condition is satisfied. Let $w \in C^2(\Omega) \cap C(\overline \Omega)$ satisfy
$$
-\Delta w \ge 0 \quad \text{in }\Omega,\qquad w \ge w(x_0) \quad \text{in }\Omega.   
$$
If $w$ is not constant in $\Omega$, then $\frac{\partial w}{\partial \eta}(x_0)<0$ 
for any outward directional derivative at $x_0$ when it exists. 
\end{Lem}

The following Proposition is a variant of the Hopf boundary lemma for a half-ball which also yields information on a 
`tangential' derivative at the corner points. 

\begin{Lem}
\label{sec:appendix-1}
Let $B_+:= \{x \in \R^N \::\:|x| \le 1,\:x_N> 0\}$. Suppose that 
$w \in C^2(\overline{B_+})$ satisfies 
$$
-\Delta w \ge 0 \quad \text{on $B_+$}, \qquad w = 0\quad \text{on $\Sigma_1$},\qquad 
\frac{\partial w}{\partial \nu}=0 \quad \text{on $\Sigma_2$,}   $$
where $\Sigma_1= \{x \in \partial B_+:\: x_N=0\}$, $\Sigma_2= 
\{x \in \partial B_+:\:x_N>0\}$, and $\nu$ is the outward normal on $\Sigma_2$. If $w \not \equiv 0$ in $B_+$, then 
$$
w >0\quad  \text{in }B_+\qquad \text{and}\qquad \frac{\partial w}{\partial x_N} >0 \quad \text{on }\Sigma_1.
$$
\end{Lem}

\begin{proof}
In the following, we write $B_r(y)$ for the closed ball of radius $r$ centered at $y \in \R^N$.
Since $w \not \equiv 0$, the maximum principle implies that $w$ cannot achieve its minimum in 
$\innt(B_+)$. Moreover,
by Lemma~\ref{sec:appendix-2} and the boundary condition 
$\frac{\partial w}{\partial \nu}=0$ on $\Sigma_2$, 
$w$ cannot achieve its minimum on $\Sigma_2$. Hence $w>0$ in $B_+$, and Lemma~\ref{sec:appendix-2} yields 
$\frac{\partial w}{\partial x_N} >0$ for 
$x \in \Sigma_1,\: |x| <1$, since at these boundary points the interior sphere condition is satisfied. It remains to prove $\frac{\partial w}{\partial x_N} >0$ for $x \in \Sigma_1$ with $|x|=1$. Without loss of generality, we only consider $x=e_1=(1,0,\dots,0)$. 
Let $\bar x = (1,0,\dots,0,1)$. We consider the functions $\phi,\psi: \R^N \setminus \{0\} \to \R$ defined by
$$
\phi(x)= \exp \bigl(-\alpha |x-\bar x|^2\bigr)-e^{-\alpha},\qquad 
\psi(x)= \phi\left(\textstyle{\frac{x}{|x|^2}}\right),
$$
where $\alpha>0$ will be fixed later. Then $\phi \equiv 0$ on $\partial B_1(\bar x)$. Moreover, 
$\partial B_1(\bar x)$ is invariant under the reflection 
$x \mapsto \frac{x}{|x|^2}$. Indeed, 
$|x-\bar x|^2=1$ implies
\begin{align*}
\Bigl |\frac{x}{|x|^2} -\bar x \Bigr|^2 &= \frac{1}{|x|^2}-\frac{2x\cdot\bar{x}}{|x|^2}+2=
\frac{|x|^2-2x\cdot\bar{x}+|\bar{x}|^2-1}{|x|^2}+1=\frac{|x-\bar x|^2-1}{|x|^{2}}+1=1.
\end{align*}
As a consequence, $\psi$ also vanishes on 
$\partial B_1(\bar x)$. We note that
$$\frac{\partial\phi}{\partial x_N}(e_1)=\left.\frac{d\phi}{ds}(\bar{x}+se_N)\right|_{s=-1}
=\left.\frac{d}{ds}\bigl(e^{-\alpha s^2}-e^{-\alpha}\bigr)\right|_{s=-1}=
2\alpha e^{-\alpha}>0
$$
and
\begin{align}
\label{eq:34}
\frac{\partial\psi}{\partial x_N}(e_1)&=\left.\frac{d}{d\theta}\psi(\cos\theta,0,\ldots,0,\sin\theta)\right|_{\theta=0}
=\left.\frac{d}{d\theta}\phi(\cos\theta,0,\ldots,0,\sin\theta)\right|_{\theta=0}\nonumber\\
&=\frac{\partial\phi}{\partial x_N}(e_1)>0.
\end{align}
Next we compute $\Delta \phi(e_1)$. For this we 
use spherical coordinates $(r,\eta)$ with center at $\bar{x}$:
$r=|x-\bar{x}| \in (0,\infty)$ and 
$\eta=\frac{x-\bar{x}}{|x-\bar{x}|} \in \partial B_1(0)$.
Let $\bar{\phi}(r,\eta):=\phi(x)$ and $\bar{\psi}(r,\eta):=\psi(x)$.
The Laplacian of a function $u(x)=\bar{u}(r,\eta)$ is given by 
\begin{equation}\label{eq:100}
\Delta u= \frac{\partial^2\bar{u}}{\partial r^2}+\frac{N-1}{r}\frac{\partial\bar{u}}{\partial r} +\Delta_\eta\bar{u},
\end{equation}
where $\Delta_\eta$ stands for Laplace-Beltrami operator on the 
sphere $\{x\in\Rb^N: |x-\bar{x}|=1\}$. 
Since $\phi$ vanishes on this sphere,
\begin{align*}\Delta\phi(e_1)&=\left[
\frac{\partial^2\bar{\phi}}{\partial r^2}+\frac{N-1}{r}\frac{\partial\bar{\phi}}{\partial r}\right]
(1,-e_N)\\
&=\left.\left[\frac{d^2}{dr^2}\bigl(e^{-\alpha r^2}-e^{-\alpha}\bigr)+
\frac{N-1}{r}
\frac{d}{dr}\bigl(e^{-\alpha r^2}-e^{-\alpha}\bigr)\right]\right|_{r=1}\\
&=
2\alpha e^{-\alpha}( 2 \alpha-N)>0,
\end{align*}
for $\alpha>\frac{N}{2}$.
In order to compute $\Delta\psi(e_1)$, 
we observe that by (\ref{eq:34}),  
$$\frac{\partial\bar{\psi}}{\partial r}(1,-e_N)=-
\frac{\partial\psi}{\partial x_N}(e_1)=-
\frac{\partial\phi}{\partial x_N}(e_1)=
\frac{\partial\bar{\phi}}{\partial r}(1,-e_N).$$
As, by a short calculation,
\begin{align*}
\frac{\partial^2\bar{\psi}}{\partial r^2}(1,-e_N)&=\left.\frac{d^2}{d\theta^2}\psi(\cos\theta,0,\ldots,0,\sin\theta)\right|_{\theta=0}
=\left.\frac{d^2}{d\theta^2}\phi(\cos\theta,0,\ldots,0,\sin\theta)\right|_{\theta=0}\\ &=
\frac{\partial^2\bar{\phi}}{\partial r^2}(1,-e_N)
\end{align*}
and $\psi$ also vanishes
on the sphere $\{x\in\Rb^N: |x-\bar{x}|=1\}$, it follows from (\ref{eq:100}) that
$$ 
\Delta \psi(e_1)= \Delta \phi(e_1).
$$
Now put $z= \phi +\psi$. Then, by construction,
$\frac{\partial z}{\partial \nu}= 0$ on $\partial B_1(0)$. 
Let $\delta>0$ be such that  
$$
\Delta z >0 \qquad \text{on }B_\delta(e_1).
$$
Since $w>0$ on $B_+$, there is $\eps>0$ such that 
$w \ge \eps z$ on $\partial B_\delta(e_1) \cap B_1(\bar x) \cap B_+$. We now consider the set $D:= B_\delta(e_1) \cap B_1(\bar x) \cap B_+$. 
The function $\tilde{w}:\overline{B_+}\setminus\{0\}\to\Rb$ defined by
$\tilde w := w- \eps z$ satisfies
$$
\begin{aligned}
\Delta \tilde w &< 0 &&\qquad \text{in }D,\\
\tilde w &\ge  0 &&\qquad \text{on }[\partial B_\delta(e_1) \cap B_1(\bar x) \cap B_+] \cup [B_\delta(e_1) \cap \partial B_1(\bar x) \cap B_+],\\
\frac{\partial \tilde w}{\partial \nu}&= 0&&\qquad \text{on }B_\delta(e_1) \cap B_1(\bar x) \cap \partial B_+.
\end{aligned}
$$
By similar arguments as above, $\tilde w$ can neither 
achieve 
its minimum on $\innt(D)$ nor on $\Sigma_2$. Since  
$\tilde w \ge 0$ on the remaining parts of $\partial D$, we conclude $\tilde w \ge 0$ on $D$. Since $\tilde w(e_1)= 0$, this implies 
$\frac{\partial \tilde w}{\partial x_N} (e_1) \ge 0$. Hence 
$$   
\frac{\partial w}{\partial x_N}(e_1) \ge \eps \frac{\partial z}{\partial x_N}(e_1) >0
$$
by (\ref{eq:34}), as claimed.
\end{proof}
\begin{Rem}
\label{sec:appendix-3}
An analogue of\/ {\rm Lemma~\ref{sec:appendix-1}} holds for a half annulus.
\end{Rem}


\begin{thebibliography}{20}
\bibitem{adimurthi.mancini} Adimurthi and Mancini, G.,
{\it The Neumann problem for elliptic equations with critical nonlinearity.} A tribute in honour of G. Prodi, Scoula Norm. Sup. Pisa (1991), 9--25.

\bibitem{AP3} Aftalion, A.\ and Pacella, F.,
{\it Qualitative properties of nodal solutions of semilinear elliptic 
equations in radially symmetric domains.}\/
C.\ R.\ Math.\ Acad.\ Sci.\ Paris 339, No.\ 5 (2004), 339--344.

\bibitem{bartsch.weth1} Bartsch, T. and Weth, T., {\em A note on additional properties of sign
  changing solutions to superlinear elliptic
  equations.} Topol. Methods Nonlinear Anal. 22, No.\ 1 (2003), 1--14.

\bibitem{BWW} Bartsch, T., Weth, T.\ and Willem M.,
{\it Partial symmetry of least energy nodal solutions
to some variational problems.} J.\ Anal.\ Math. 96 (2005), 1--18.


\bibitem{baernstein.taylor} Baernstein A., Taylor, B.A., \textit{Spherical rearrangements, subharmonic functions, and $\sp*$-functions in
$n$-space.} Duke Math. J. 43, No.\ 2 (1976), 245--268.


\bibitem{BB} Bañuelos, R.\ and Burdzy, K.,
{\it On the ``hot spots" conjecture of J. Rauch.}
J. Funct. Anal. 164, No.\ 1 (1999), 1--33.


\bibitem{belloni-kawohl} Belloni, M. and Kawohl, B.,  {\it A symmetry problem related to Wirtinger's and Poincaré's inequality}, 
J. Differ. Equations 156, No.\ 1 (1999), 211--218.


\bibitem{brezis.lieb} Brezis, H. and Lieb, E., {\it A relation between pointwise convergence of functions
and convergence of functionals}, Proc. Amer. Math. Soc. 88 (1983), 486--490.

\bibitem{brock} Brock, A., {\it Symmetry and monotonicity of solutions to some variational problems in cylinders and annuli.}\ Electron. 
J. Diff. Equ.\ 2003, No. 108 (2003), 1--20.


\bibitem{brock.solynin:2000}
Brock, F. and Solynin, P., {\em An approach to symmetrization via
  polarization.} Trans. Am. Math. Soc. 352, No.\ 4 (2000),
  1759--1796. 
 

\bibitem{bus-kon-naz} Buslaev, A.P., Kondratiev, V.A. and  Nazarov, A.I.,
{\em On a family of extremum problems and the properties of an integral.} Math. Notes 64, No.\ 6 (1998), 719--725.

\bibitem{CL} Carlen, E.A. and  Loss, M.,
{\it On the minimization of symmetric functionals.}\/ 
Rev. Math. Phys. 6, No.\ 5a (1994) 1011--1032.

\bibitem{casten.holland:1978} Casten, R.G. and Holland, C.J., 
{\it Instability results for reaction diffusion equations with Neumann boundary conditions.} J. Differ. Equations 27 (1978), 266--273.

\bibitem{cherrier} 
  Cherrier, P., {\it Meilleures constantes dans les in\'egalit\'es relatives aux espaces de Sobolev.}\ Bull. Sci. Math., $2^e$ s\'erie. {108} (1984), 227--262.

\bibitem{CNV}
Cordero-Erausquin, D., Nazaret, B. and  Villani, C.
{\it A mass-transportation approach to sharp Sobolev and Gagliardo-Nirenberg inequalities.}\/ 
Adv. Math. 182, No. 2  (2004), 307--332.

\bibitem{DGS}
Dacorogna, B., Gangbo, W.\ and Subía, N.,
{\it Sur une généralisation de l'inégalité de Wirtinger.}\/
Ann.\ Inst.\ H.\ Poincaré, Anal.\ Non Linéaire 9, No.\ 1 (1992), 29--50.

\bibitem{dancer} Dancer, E.N., {\it Real Analyticity and 
non-degeneracy}, Math. Ann. 325, No.\ 2 (2003), 369--392.


\bibitem{egorov-alt} 
Egorov, Y.V., {\em On a Kondratiev problem,}  C. R. Acad. Sci. Paris Sér. I Math.  324, No.~5  (1997),  
503--507.

\bibitem{GT} Gilbarg, D. and Trudinger, N.S., \textit{Elliptic partial differential equations of second order}, Grundlehren der mathematischen Wissenschaften 
{224}, Springer 1977.

\bibitem{giusti-new} Giusti, E., Direct Methods in the Calculus of Variations, World Scientific 2003.

\bibitem{kawohl} Kawohl, B., {\it Rearrangements and Convexity of Level Sets in PDE}, 
Lecture Notes in Mathematics 1150, Springer-Verlag 1985.


\bibitem{kawohl1} Kawohl, B., {\em Symmetry results for functions yielding best constants in Sobolev-type inequalities.} 
Discrete Contin. Dyn. Syst. 6, No.\ 3 (2000), 683--690.


\bibitem{jerison} Jerison, D., 
{\em Locating the first nodal line in the Neumann problem.}
Trans. Am. Math. Soc. 352,  No.\ 5 (2000), 2301--2317.


\bibitem{lin}
Lin, C.S., {\em Uniqueness of least energy solutions to a semilinear elliptic equation in $R\sp 2$.}
Manuscr. Math. 84, No.\ 1 (1994), 13--19.

\bibitem{MV}
Maggi, F. and Villani, C., {\it Balls have the worst best Sobolev inequalities.}\/ J. Geom. Anal.
15, No.\ 1 (2005), 83--121.

\bibitem{maz'ya} Maz'ya, V.G., {\it Sobolev Spaces.} Springer-Verlag 1985.

\bibitem{MP85} M\"uller-Pfeiffer, E., {\em On the number of nodal domains for elliptic 
differential operators.} J. Lond. Math. Soc. II Ser. {31}, (1985), 91--100.


\bibitem{nazarov-new} Nazarov, A.I.,
{\em On an exact constant in the generalized Poincaré inequality.}
J. Math. Sci., New York 112, No.\ 1 (2002), 4029--4047.


\bibitem{nazarov-2d} Nazarov, A.I.,
{\em On the ``one-dimensionality" of the extremal in the Poincaré inequality in a square.} 
J. Math. Sci., New York 109, No.\ 5 (2002), 1928--1939.


\bibitem{payne.weinberger} Payne, L.E. and Weinberger, H.F., {\em An optimal Poincar\'e inequality for convex domains.} 
Arch. Ration. Mech. Anal. 5 (1960), 286--292.

\bibitem{Ren-Wei}
Ren, X.F. and Wei, J.C., 
{\it On a two-dimensional elliptic problem with large exponent in nonlinearity.}\/ Trans. Am. Math. Soc. 343, No.\ 2 (1994), 749--763.

\bibitem{simon} Simon, B., {\it Schrödinger Semigroups}, Bull. Am. Math. Soc. New Ser. 7 (1982), 447--526.

\bibitem{SW} Smets, D.\ and Willem, M., {\it Partial
symmetry and asymptotic behavior for some elliptic
variational problems.}\/ Calc.\ Var.\ Partial Differ. Equ. 18, No.\ 1
(2003), 37--75.

\bibitem{szegoe} Szeg\"o, G., {\em Inequalities for certain eigenvalues of a membrane of given area.}
J. Ration. Mech. Anal. 3 (1954), 343--356.

\bibitem{wang} Wang, X.J., {\em Neumann problem of semilinear elliptic equations involving critical Sobolev exponents.} J. Differ. Equations 93 (1991), 283--310.

\bibitem{Weinberger} Weinberger, H.F.,
{\it An isoperimetric inequality for the $n$-dimensional free membrane problem.}
J.\ Ration. Mech.\ Anal.\ 5 (1956), 633--636.

\bibitem{Z2} Zhu, M.,\
{\it Extremal functions of Sobolev-Poincaré inequality.}\/
Nonlinear Evolution Equations and Dynamical Systems, 
World Sci.\ Publishing (2003), 173--181.

\bibitem{Z1} Zhu, M.,\
{\it On the extremal functions of Sobolev-Poincaré inequality.}\/
Pac. J.\ Math.\ 214, No. 1 (2004), 185--199.

\bibitem{Z3} Zhu, M.,
{\it Sharp Poincaré-Sobolev inequalities and the shortest length 
of simple closed geodesics on a topological two sphere.}\/
Commun.\ Contemp.\ Math.\ 6, No. 5 (2004), 781--792.

\end{thebibliography}
\end{document}